\documentclass[12pt, letter]{article}

\usepackage{thkim1011}
\usepackage{tikz}
\usepackage{tikz-cd}
\usepackage{xcolor}
\usepackage{indentfirst}
\usepackage{bm}

\renewcommand{\C}{\mathcal{C}}
\renewcommand{\S}{\mathcal{S}}

\renewcommand{\P}{\mathcal{P}}
\newcommand{\cht}{\text{coht}}

\let\svwidehat\widehat
\renewcommand\widehat[1]{\ifx\relax#1\relax\svwidehat{\phantom{x}}\else\svwidehat{#1}\fi}

\renewcommand{\hat}{\widehat}

\title{Structure of Spectra of Precompletions}
\author{Erica Barrett, Emil Graf, S. Loepp, Kimball Strong, and Sharon Zhang}

\begin{document}
\maketitle

\begin{abstract}
    Let $T$ be a complete local (Noetherian) ring, and let $A$ be a local subring of $T$ such that the completion of $A$ with respect to its maximal ideal is $T$. We investigate the possible structures of the partially ordered set $\Spec (A)$. Specifically, we explore the minimal prime ideals of $A$ and their formal fibers, the maximal chains of prime ideals in $A$, and the number of prime ideals in $A$ containing combinations of minimal prime ideals of $A$.
\end{abstract}

\section{Introduction}
\par Complete local rings are fully characterized by Cohen's Structure Theorem, but we have a poor understanding of local rings which are not complete. Thus, the relationship between a local ring and its completion is of prime importance in understanding the ring itself.
In studying this relationship, we start with a complete local ring $T$ and explore the properties of rings whose completion is $T$. In other words, we ask: given a complete local ring $T$, when does there exist a local ring $A$ such that the completion of $A$ is $T$ and $A$ satisfies a specific property?
We call this local ring $A$ a ``precompletion" of $T$. Note that a particular complete local ring $T$ can, and often does, have multiple precompletions, and each of these precompletions can satisfy a variety of different properties.
\par The question of precompletion properties has already been answered for some particular properties. For example, in \cite{lech}, Lech characterizes complete local rings that are completions of domains, and in \cite{heitmann93}, Heitmann does the same for completions of Unique Factorization Domains. In this paper, we focus on properties concerning the prime spectrum of a precompletion $A$ of $T$.
\par The prime spectrum (or simply spectrum) of any ring can be viewed as a partially-ordered set with respect to inclusion. 
We can (informally) organize the structure of the spectrum in two ways---\textit{horizontally}, which focuses on those prime ideals which are incomparable, and \textit{vertically}, which focuses on chains of prime ideals that contain one another. 
For instance, the cardinality of $\Min(A)$ is a horizontal property, whereas dimension is vertical. 
With these structural aspects in mind, we explore the following type of question: given a complete local ring $T$ and a potential property of a partially ordered set, when is there a precompletion $A$ of $T$ such that $\Spec(A)$, when viewed as a partially ordered set, has this property?  
For example, if $T$ is the complete local ring $\frac{\mathbb{Q}[[x,y,z]]}{(xyz)}$, we ask if $T$ can be the completion of a local ring $A$ such that $A$ has exactly two minimal prime ideals $q_1$ and $q_2$, and exactly two height one prime ideals that contain both $q_1$ and $q_2$.  In fact, in example Example \ref{standardexample}, we show that such a ring $A$ exists, and that the spectrum of the ring $A$ can be countable or it can be uncountable.
\par First, we study the horizontal structure of spectra. Given a local ring $R$ and prime ideal $P$ of  $R$, we will consider the formal fiber of $R$ at $P$ to be the set of prime ideals $Q \in \Spec(\hat{R})$ such that $Q \cap R = P$ (here $\hat{R}$ denotes $I$-adic completion for some ideal $I$ of $R$). While this is not quite the technical definition of a formal fiber, it suffices for our purposes.
We build off of results from \cite{small09} and \cite{byron} to examine the formal fibers of certain prime ideals of precompletions, paying particular attention to the formal fibers of the minimal prime ideals. We show the existence of an upper bound on the cardinalities of specific sets of prime ideals in the spectra of possible precompletions. Moreover, by adapting the constructions in \cite{small09} and \cite{byron}, we construct a precompletion in which this upper bound is achieved.
\par In the subsequent section, we examine the vertical structure of precompletion spectra. Specifically, we consider the relationship between the lengths of maximal chains of prime ideals in $\Spec(T)$ and the lengths of maximal chains of prime ideals in $\Spec(A)$, where $A$ is a precompletion of $T$.
Referencing the work of \cite{small17} and the construction from the previous section, we show that the possible lengths of maximal chains in $\Spec(A)$ are determined by the lengths of maximal chains in $\Spec(T)$ and that we can construct a precompletion of $T$ with maximal chains having lengths equal precisely to the lengths of the maximal chains in $\Spec(T)$.
\par All rings we consider are commutative with unity element $1$.  Throughout this paper, ``local" describes a Noetherian ring with exactly one maximal ideal, and ``quasi-local" describes a ring (not necessarily Noetherian) with one maximal ideal. The ordered pair $(A, M)$ denotes a local ring $A$ with maximal ideal $M$.  The specific ordered pair $(T,M)$ denotes a complete local ring $T$ with maximal ideal $M$.  Finally, if $(A, M)$ is a local ring, we use $\widehat{A}$ to denote the completion of $A$ at $M$.

\section{Preliminaries}

In Sections Two and Three, we extend a result from \cite{byron} that generalizes results from \cite{small09}. Both results in  \cite{byron} and \cite{small09} are concerned with prescribing the formal fibers of the minimal prime ideals of a reduced local ring $A$ whose completion is $T$. We generalize this result so that $A$ need not be reduced, though in doing so some of our conditions devolve from ``necessary and sufficient'' to simply ``sufficient.'' 
\par The setting of the main result consists of a complete local ring $T$, a set of prime ideals $\C$ of $T$ and a partition $\{ \C_i \}_{i = 1}^{m}$ on $\C$ that simultaneously partitions the minimal prime ideals of $T$.
We then find a precompletion $A$ of $T$ such that, for all $i$ and for all $P,P' \in \C_i$, we have $P \cap A = P' \cap A$ and $P \cap A$ is a minimal prime ideal of $A$.
In addition, our construction allows for some control of the prime ideals (of height at least one) containing various combinations of minimal prime ideals.
\par Throughout this section and the next one we will use the following example to illustrate many concepts we introduce. \\

\begin{ex} \label{spec diagram ex}
We consider the complete local ring $\frac{\mathbb{Q}[[x,y,z]]}{(xyz)}$, and view its prime spectrum as a partially ordered set:
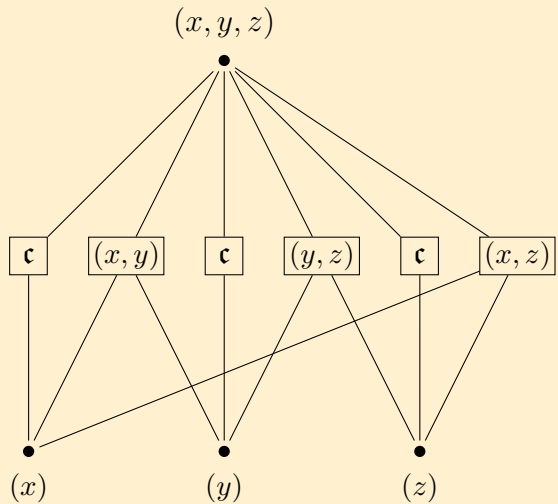
\begin{center}
    \begin{tikzpicture}[scale = 1.3]
        \node[label=above:{$(x,y,z)$}] (top) at (9, 3) {};
        
        \node[shape=rectangle,draw=black,minimum size=0.5cm,label=center:{$\mathfrak{c}$}] (mid1) at (7, 1) {};
        \node[shape=rectangle,draw=black,minimum size=0.5cm,label=center:{$\mathfrak{c}$}] (mid2) at (9, 1) {};    \node[shape=rectangle,draw=black,minimum size=0.5cm,label=center:{$\mathfrak{c}$}] (mid3) at (11, 1) {};
        
        \node[shape=rectangle,draw=black, minimum width=1.0cm, minimum height=.5cm,label=center:{\small$(x,y)$}] (shared1) at (8,1) {};
        \node[shape=rectangle,draw=black,minimum width=1.0cm,minimum height = .5cm, label=center:{\small$(y,z)$}] (shared2) at (10,1) {};
        \node[shape=rectangle,draw=black,minimum width=1.0cm, minimum height=.5cm, label=center:{\small{}$(x,z)$}] (shared3) at (12,1) {};
        
        \node[label=below:{\small$(x)$}] (min1) at (7, -1) {};
        \node[label=below:{\small$(y)$}] (min2) at (9, -1) {};
        \node[label=below:{\small$(z)$}] (min3) at (11,-1) {};
        
        \foreach \n in {top,min1,min2,min3}
            \node at (\n)[circle,fill,inner sep=1.5pt]{};
        
        \foreach \n in {shared1,shared2,shared3,mid1,mid2,mid3}
        \path [-] (top) edge node[left] {} (\n);
        
        \path [-] (mid1) edge node[left] {} (min1);
        \path [-] (mid2) edge node[left] {} (min2);
        \path [-] (mid3) edge node[left] {} (min3);
        
        \path [-] (shared1) edge node[left] {} (min1);
        \path [-] (shared1) edge node[left] {} (min2);
        
        \path [-] (shared2) edge node[left] {} (min2);
        \path [-] (shared2) edge node[left] {} (min3);
        
        \path [-] (shared3) edge node[left] {} (min3);
        \path [-] (shared3) edge node[left] {} (min1);
    \end{tikzpicture}
\end{center}
Here the boxes denote collections of prime ideals; e.g.~there are $\mathfrak{c}$ (the cardinality of the real numbers) prime ideals $P$ with $( x ) \subsetneq P \subsetneq (x,y,z)$, there are exactly two prime ideals which contain both $(x)$ and $(y)$ (the maximal ideal and $(x,y)$), etc.
\end{ex}

\par In both \cite{small09} and \cite{byron}, the authors use what they call a \textit{feasible partition} to control the formal fibers of the minimal prime ideals of $A$. 
We alter their definition slightly. We require that conditions (i) and (ii) are satisfied for elements of $\Min(T)$ instead of elements of $\Ass(T)$, and therefore refer to our partition as a ``minfeasible partition." In addition, the definition is modified (with respect to \cite{small09}) such that the set $\C$ is allowed to be countably infinite. 
In \cite{byron}, it is shown that all the lemmas still hold when $\C$ is countably infinite, so we only have to consider the first modification.

\par Note that in \cite{small09} and \cite{byron}, the authors only consider partitions with $m \ge 2$ subcollections because the case of one subcollection was already fairly studied (domains being the main example), but all the proofs work the same with one subcollection except for a trivial modification to the proof of Lemma \ref{Byron 2.12}, and one subcollection is useful for our work, so we will allow this.

\begin{dfn}\label{minfeasible}
Let $(T, M)$ be a complete local ring. Let $\C$ be a countable collection of incomparable non-maximal prime ideals of $T$, which we partition into $m \geq 1
$ subcollections $\C_1, \ldots, \C_m$. We call $\mathcal{P} = (\C, \{ \C_i \}_{i = 1}^{m})$ a \textbf{minfeasible partition on $\C$} (or simply a \textbf{minfeasible partition}) if, for each $Q$ in $\Min(T)$, $\mathcal{P}$ satisfies the following conditions:
\begin{enumerate}[label={(\roman*)}]
    \item $Q \subseteq P$ for at least one $P \in \C$.
    \item There exists exactly one $k$ such that, whenever $Q \subseteq P$ with $P \in \C$, we have $P \in \C_k$.
\end{enumerate}
We will say that a prime ideal $\overline{Q} \in \Spec(T)$ is \textbf{under} $\C_i$ if $\overline{Q} \subseteq P $ for some $P \in \C_i$, and similarly $\overline{Q} \in \Spec(T)$ is \textbf{over} $\C_i$ if $\overline{Q} \supsetneq P$ for some $P$ under $\C_i$ and if in addition $\overline{Q}$ itself is not under $\C_i$. 
\end{dfn}
\noindent Conditions (i) and (ii) of Definition \ref{minfeasible} can be rephrased as ``each $Q$ in $\Min(T)$ is under exactly one $\C_k$.'' In this way, our partition on $\C$ also partitions the minimal prime ideals of $T$.

\begin{ex} \label{minfeasible ex}
Using the complete local ring $\frac{\mathbb{Q}[[x,y,z]]}{( xyz )}$ from Example \ref{spec diagram ex}, we can create a minfeasible partition $\mathcal{P} = (\C, \{ \C_1, \C_2 \})$, where $\C =  \{ ( x ), ( y ), ( z ) \}$, $\C_1 =  \{ ( x ), ( y ) \}$, and $\C_2 =  \{( z ) \}$.
\par In this example, the prime ideal $(x, y)$ is over $\C_1$ but not over $\C_2$, and the prime ideals $(x, z)$ and $(y, z)$ are both over $\C_1$ and $\C_2$. The minimal prime ideals $(x)$ and $(y)$ are under $\C_1$, and the minimal prime ideal $(z)$ is under $\C_2$.
\end{ex}

Given a minfeasible partition, we want to construct a precompletion $A$ of $T$ such that $A$ has exactly $m$ minimal prime ideals $q_1,q_2, \ldots ,q_m$ and such that if $P \in \Spec(T)$ is under $\C_i$ then $P \cap A = q_i$.  A Min-T subring of $T$ is a subring that satisfies our desired intersection properties for elements of our minfeasible partition. Of main importance to us will be junior Min-T subrings, which are Min-T subrings with an additional cardinality restriction.
We will construct the precompletion $A$ as a junior Min-T or Min-T subring of $T$. Note that in order to control the minimal prime ideals in $A$, we will at minimum need a subring for which conditions (ii) and (iii) in the following definition hold. \\

\begin{dfn}\label{Min-T subring}
    Let $(T, M)$ be a complete local ring, and $(\C, \{ \C_i \}_{i = 1}^m)$ a minfeasible partition. A \textbf{semi-junior Min-T subring} (or semi-JMT subring) of $T$ is a quasi-local subring $(R, R \cap M)$  of $T$ such that the following conditions hold:
    \begin{enumerate}[label={(\roman*)}]
        \item $R$ is infinite
        \item For each $P \in \C$, $P \cap R = Q \cap R$ for any $Q \in \Min(T)$ with $Q \subseteq P$
        \item For each $\C_i$, if $P, P' \in \C_i$, then $P \cap R = P' \cap R$
        \item If $\C_i \neq \{ Q^*\}$ where $Q^* \in \Min(T)$, then for each $P \in \C_i$ and $r \in P \cap R$, we have $\Ann_T(r) \nsubseteq P$ 
        \item $|R| < |T|$
    \end{enumerate}
    A semi-JMT subring of $T$ is called a \textbf{junior Min-T subring} of $T$, or JMT subring, if it also satisfies the converse of condition (iii) above. 
    That is, for any $P, P' \in \C$, if $P \cap R = P' \cap R$ then $P, P' \in \C_i$ for some $i$. 
    Throughout this paper, we will usually consider this property as a part of condition (iii) in the definition of a JMT subring.
    Lastly, if $R$ satisfies properties (i)-(iv), but not (v), of the definition of a JMT subring, then we call $R$ simply a \textbf{Min-T subring} of $T$.
\end{dfn}

Since the definitions of minfeasible partitions, JMT subrings and semi-JMT subrings are very similar to the definitions of feasible partitions, SIP-subrings and semi-SIP-subrings in \cite{small09}, many of our lemmas are analogs of lemmas in \cite{small09} and \cite{byron} and can be proved in the exact same way, occasionally substituting $\Min(T)$ for $\Ass(T)$ where appropriate. 
There are two differences we need to consider. One is that we allow $\Ann_T(r) \subseteq P$ for $r \in P \cap R$ if $\C_i = \{Q^*\}$ for some $Q^* \in \Min (T)$. This modification is trivial, as condition (iv) is only used to ensure that prime ideals under a certain $\C_i$ have the same intersection with our JMT subrings as we adjoin more elements; since there is only one prime ideal under such a $\C_i$, this will always be the case (we will still note this when necessary, specifically in Lemma \ref{Byron 2.8}). 
The only other difference is that there may be $Q \in \Ass(T)$ such that $Q \nsubseteq P$ for all $P \in \C$, an event which is allowed in our construction but violates the definition of a feasible partition in \cite{small09} and \cite{byron}. We will therefore need to consider reproving lemmas in \cite{small09} and \cite{byron} that use the following condition: 
\[
  \text{For all } Q \in \Ass(T), \, Q \subseteq P \text{ for some } P \in \C.\tag{$\ast$}
\]
The main consequence of this change is that Min-T subrings are not necessarily reduced. \\

\begin{rmk}
Note that given some $\C_k$ in a minfeasible partition $\{ \C, \{ \C_i \}_{i = 1}^m \}$ and any semi-JMT or JMT subring $A$ (or any subring satisfying conditions (ii) and (iii) of Definition \ref{Min-T subring}), it is true that $Q \cap A = Q' \cap A$ for any $Q, Q'$ under $\C_k$. 
In order to simplify our proofs, we maintain the following abuse of notation found in \cite{small09} and \cite{byron}: given any prime ideal $P$ of $T$, we write $P \cap A \supseteq \C_k \cap A$ in place of $P \cap A \supseteq Q \cap A$ for all $Q$ under  $\C_k$.
\end{rmk}

\par The following two lemmas describe restrictions on the partial order of prime ideals in $\Spec(T)$ with respect to inclusion, given their relationship in a local ring $A$ where $\widehat{A} = T$. \\

\begin{lem}\label{prelim}
Let $(T, M)$ be the completion of a local ring $(A, A \cap M)$ and $P$ a prime ideal of $A$. Then for $Q \in \Min(PT)$, we have $Q \cap A = P$.
\end{lem}

\begin{proof}
Consider the domain $A/P$, and note that $\widehat{A/P} = T/PT$. The minimal prime ideals of $T/PT$ are precisely the images of $Q \in \Min(PT)$ in $T/PT$, and their intersection with $A/P$ must be the zero ideal of $A/P$. Thus $Q \cap A = P$ for all $Q \in \Min(PT)$.
\end{proof}

\begin{lem}\label{General Poset Property}
Let $A$ be a local ring, with $\widehat{A} = T$. For any $P_1, P_2 \in \Spec(T)$, $P_1 \cap A \supseteq P_2 \cap A$ if and only if $P_1$ contains a prime ideal $Q$ of $T$ such that $Q \cap A = P_2 \cap A$.
\end{lem}

\begin{proof} 
The `if' direction is evident. For the `only if' direction: Let $P_1 \cap A \supseteq P_2 \cap A$. 
Then $P_1 \supseteq (P_2 \cap A)T$, so $P_1$ contains a minimal prime ideal over $(P_2 \cap A)T$, say $Q$. Then $Q \cap A = P_2 \cap A$ by Lemma \ref{prelim}, as desired.
\end{proof}

The following theorem is an abstraction of the previous lemma; it is not hard to show that the two are equivalent. 
Though for practical purposes we will refer to Lemma \ref{General Poset Property}, we include the following theorem for its simplicity. 
Essentially, the theorem states that given any local ring $A$ with $\widehat{A} = T$, the poset $\Spec(A)$ is determined in the ``most free'' way possible by the formal fibers. 
This gives intuition as to the result of Theorem \ref{success}, and furthermore makes clear that the only way to control the properties of $\Spec(A)$ as a poset is by controlling the formal fibers of its elements. 
This motivates the general approach of constructing precompletions with prescribed formal fibers.  
\\

\begin{thm}\label{Abstract View}
    Let $A$ be a local ring with $\widehat{A} = T$. Let $\phi : \Spec(T) \to \Spec(A)$ be the map induced by the completion map $A \to T$. Define an equivalence relation $\sim$ on $\Spec(T)$ by 
    \[
    P_1 \sim P_2 \iff \phi(P_1) = \phi(P_2).
    \]
    Then $\Spec(A)$ is order-isomorphic to $\Spec(T)/\sim$.
\end{thm}

\begin{proof}
    Recall that for a poset $(X, \le)$, given an equivalence relation $\sim$, we define $X/\sim$ to be the partially ordered set where the underlying set is $X/\sim$ (as a set) and where the new partial order $\le_{\sim}$ is given by 
    \[
    [a] \le_\sim [b] \iff \text{there exist } a' \in [a] \text{, } b' \in [b] \text{ such that } a' \le b'.
    \]
    Applying this definition to the result in Lemma \ref{General Poset Property} and using the fact that $\phi$ is surjective gives us the desired result. 
\end{proof}

\par Below, we define two special subsets of $\Spec(T)$ and $\Spec(A)$. These subsets are the primary components of the relationship between the two spectra induced by the natural map from $\Spec(T)$ to $\Spec(A)$. \\

\begin{dfn}
Let $T$ be a complete local ring, $\mathcal{P} = \{ \mathcal{C}, \{ \mathcal{C}_i \}_{i = 1}^m \}$ a minfeasible partition, and $A$ a precompletion of $T$ which satisfies conditions $(ii)$ and $(iii)$ of Definition \ref{Min-T subring}. For a set $X \subseteq \{ 1, \ldots, m \}$, define 
$$
\mathcal{S}_T(X) = \{ P \in \Spec(T) \mid P \text{ is over } \C_i \; \, \forall i \in X \}
$$
and
$$
\mathcal{S}_A(X) = \{ P \in \Spec(A) \mid P \supsetneq \C_i \cap A \; \, \forall i \in X \}.
$$
\end{dfn}

The following example further illustrates the sets $\mathcal{S}_T(X)$ and $\mathcal{S}_A(X)$. \\
\begin{ex} \label{split ex}
Consider the complete local ring $\frac{\mathbb{Q}[[x,y,z]]}{( pz )}$, where $p$ is a polynomial in $x$ and $y$ that factors in the ring $\mathbb{Q}[[x,y,z]]$ with prime factorization $p = p_1p_2$, but is irreducible over $\mathbb{Q}[x,y,z]$, and the minfeasible partition $\mathcal{P} = (\C, \{ \C_1, \C_2 \})$, where $\C_1 =  \{ ( p_1 ), ( p_2 ) \}$, and $\C_2 =  \{( z ) \}$. Then $A=\frac{\mathbb{Q}[x,y,z]}{( pz )}$ is a JMT subring that is a precompletion of $T$. The set $\S_T(\{1,2\}) = \{ ( p_1,z ), ( p_2,z ), ( x,y,z ) \}$ is illustrated below:
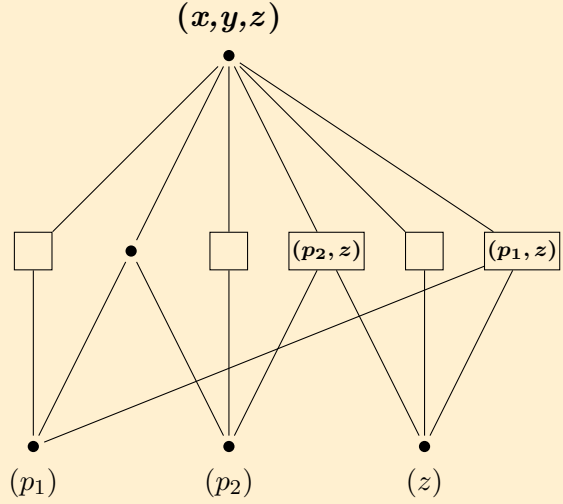
\begin{center}
    \begin{tikzpicture}[scale = 1.3]
        \node[label=above:{\textbf{(\textit{x,y,z})}}] (top) at (9, 3) {};
        
        \node[shape=rectangle,draw=black,minimum size=0.5cm] (mid1) at (7, 1) {};
        \node[shape=rectangle,draw=black,minimum size=0.5cm] (mid2) at (9, 1) {};
        \node[shape=rectangle,draw=black,minimum size=0.5cm] (mid3) at (11, 1){};
        
        \node[] (shared1) at (8, 1) {};
        \node[shape=rectangle,draw=black,minimum width=1.0cm,minimum height=.5cm,label=center:{\scriptsize\boldsymbol{$(p_2,z)$}}] (shared2) at (10, 1) {};    
        \node[shape=rectangle,draw=black,minimum width=1.0cm,minimum height=.5cm,label=center:{\scriptsize\boldsymbol{$(p_1,z)$}}] (shared3) at (12, 1) {};
        
        \node[label=below:{\small$( p_1 )$}] (min1) at (7, -1) {};
        \node[label=below:{\small$( p_2 )$}] (min2) at (9, -1) {};
        \node[label=below:{\small$( z )$}] (min3) at (11,-1) {};
        
        \foreach \n in {top,shared1,min1,min2,min3}
            \node at (\n)[circle,fill,inner sep=1.5pt]{};
        
        \foreach \n in {shared1,shared2,shared3,mid1,mid2,mid3}
        \path [-] (top) edge node[left] {} (\n);
        
        \path [-] (mid1) edge node[left] {} (min1);
        \path [-] (mid2) edge node[left] {} (min2);
        \path [-] (mid3) edge node[left] {} (min3);
        
        \path [-] (shared1) edge node[left] {} (min1);
        \path [-] (shared1) edge node[left] {} (min2);
        
        \path [-] (shared2) edge node[left] {} (min2);
        \path [-] (shared2) edge node[left] {} (min3);
        
        \path [-] (shared3) edge node[left] {} (min3);
        \path [-] (shared3) edge node[left] {} (min1);
    \end{tikzpicture}
\end{center}

The set $\S_A(1,2) = \{ ( p,z ), ( x,y,z ) \}$, is illustrated below.

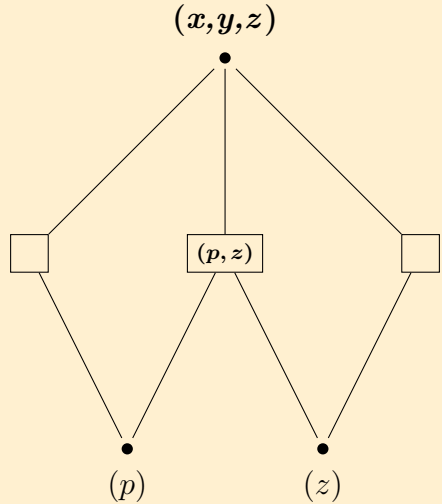
\begin{center}
    \begin{tikzpicture}[scale = 1.3]
        \node[label=above:{\textbf{(\textit{x,y,z})}}] (A) at (3, 3) {};
        \node[shape=rectangle,draw=black,minimum size=0.5cm] (B) at (1, 1) {};
        \node[shape=rectangle,draw=black,minimum height=.5cm,minimum width=1.0cm,label=center:{\scriptsize\boldsymbol{$(p,z)$}}] (C) at (3, 1) {};    
        \node[shape=rectangle,draw=black,minimum size=0.5cm] (D) at (5, 1) {};
        \node[label=below:{$( p )$}] (E) at (2, -1) {};
        \node[label=below:{$( z )$}] (F) at (4, -1) {};
            
        \foreach \n in {A,E,F}
            \node at (\n)[circle,fill,inner sep=1.5pt]{};
        
        \path [-] (A) edge node[left] {} (B);
        \path [-] (A) edge node[left] {} (C);
        \path [-] (A) edge node[left] {} (D);
        \path [-] (B) edge node[left] {} (E);
        \path [-] (C) edge node[left] {} (F);
        \path [-] (C) edge node[left] {} (E);
        \path [-] (D) edge node[left] {} (F);

    \end{tikzpicture}
\end{center}
\end{ex}

As seen in the example above, $\mathcal{S}_T(X)$ and $\mathcal{S}_A(X)$ are closely connected, which is unsurprising given the results of Lemmas \ref{prelim} and \ref{General Poset Property}. In the following theorem, we explicitly define this relationship with respect to the induced homomorphism from $\Spec(T)$ to $\Spec(A)$. \\

\begin{thm}\label{success}
Let $(T, M)$ be a complete local ring, $(\C, \{ \C_i \}_{i = 1}^m)$ a minfeasible partition of $T$, and $A$ a local ring satisfying properties (ii) and (iii) of Definition \ref{Min-T subring} and such that $\widehat{A} = T$ and $\Min(A) = \{ \C_1 \cap A, \ldots, \C_m \cap A \}$. Then for any $X \subseteq \{ 1, \ldots, m \}$, the map $P \mapsto P \cap A$ is a surjection from $\S_T(X)$ to $\S_A(X)$. Furthermore, if $P \cap A \in \S_A(X)$ then $P \in \S_T(X)$.
\end{thm}

\begin{proof}
\par Let $\Min(T) = \{ Q_1, \ldots, Q_n \}$. 
Suppose $P$ is a prime ideal of $T$ such that $P \cap A \supsetneq \C_i \cap A$ for all $i \in X$, that is, $P \cap A$ is an element of $\S_A(X)$. Using Lemma \ref{General Poset Property}, we have that, for every $i \in X$, there is some $P'_i \in \Spec(T)$ with $P'_i \subseteq P$ such that $P'_i \cap A = \C_i \cap A$. Note that $P'_i$ must contain some minimal prime ideal $Q_i$. Then $P'_i \cap A = Q_i \cap A$ because $P'_i \cap A$ is minimal, so $Q_i$ is under $\C_i$. In other words, $P$ contains an element under $\C_i$. Observe that if $P$ is under $\C_i$ then $P \cap A = \C_i \cap A$, a contradiction. Hence $P$ is not under $\C_i$, and it follows that $P$ is in $\S_T(X)$.
\end{proof}

The following corollary notes that the map $P \mapsto P \cap A$ yields a condition on the formal fibers of the prime ideals in $A$ which are minimal over sums of minimal prime ideals in $A$. 
Specifically, for $X \subseteq \{ 1,...,m \}$ with $|X| \ge 2$, we choose a minimal prime ideal $Q_i$ under $\C_i$ for all $i$ and then look at minimal prime ideals over the sum $\sum_{i \in X} Q_i$, repeating for all possible combinations of choices of $Q_i$.
\\
\begin{cor} \label{height-one map}
    Let $X \subseteq \{1,\dots, m\}$ with $|X| \ge 2$. 
    Denote by $\overline{\S}_T^{1}(X)$ the subset of $\S_T(X)$ which consists of prime ideals minimal over a sum of ideals $\sum_{i \in X} Q_i$ where each $Q_i$ is an element of $\Min(T)$ which is under $\C_i$. 
    Let $\S_T^1(X)$ be the set of minimal elements of $\overline{\S}_T^1(X)$. Denote by $\S_A^1(X)$ the subset of $\S_A(X)$ consisting of prime ideals minimal over $\sum_{i \in X} \C_i \cap A$.
    Then for every element $p \in \S_A^1(X)$, there exists $P'' \in \S_T^1(X)$ such that $P'' \cap A = p$.
\end{cor}
\begin{proof}
Let $p \in \S_A^1(X)$ for some $X \subseteq \{ 1, \ldots, m \}$. Then $p$ is minimal over $\sum_{i \in X} \C_i \cap A$, and $p \in \S_A(X)$, so, by Theorem \ref{success}, $p = P \cap A$ for some $P \in \S_T(X)$. 
Then $P \supseteq \sum_{i \in X} Q_i$ for some $Q_i \in \Min(T)$ where $Q_i$ is under $\C_i$.  Hence, $P$ contains a minimal prime ideal over $\sum_{i \in X} Q_i$, call it $P'$. 
Then $P' \in \overline{\S}_T^1(X)$, so there exists $P'' \subseteq P'$ such that $P'' \in \S_T^1(X)$. Then $P \cap A \supseteq P'' \cap A \supseteq \sum_{i \in X} \C_i \cap A$. But $P \cap A = p$, and $p$ is minimal over $\sum_{i \in X} \C_i \cap A$, so $P'' \cap A = p$.
\end{proof}
The main result of Corollary \ref{height-one map} is an upper bound on $|\S_A^1(X)|$ by $|\S_T^1(X)|$ (note that both are finite). 
Later, we will show that this upper bound is tight; that is, it can be achieved for a certain precompletion.

\begin{ex} \label{height one boxes ex}
Consider the complete local ring $\frac{\mathbb{Q}[[x,y,z,w]]}{(xy,zw)}$ and the minfeasible partition $\mathcal{P} = (\C, \{ \C_1, \C_2 \})$, where $\C_1 =  \{ ( x,z ), ( y,z ) , (x,w) \}$, and $\C_2 =  \{( y,w ) \}$. The sets $\overline{\S}_T^1(\{1,2\}) = \{ ( x,y,z,w ), ( x,y,w ), ( y,z,w ) \}$ and $\S_T^1(\{1,2\})=\{(x,y,w),(y,z,w)\}$ are illustrated below:
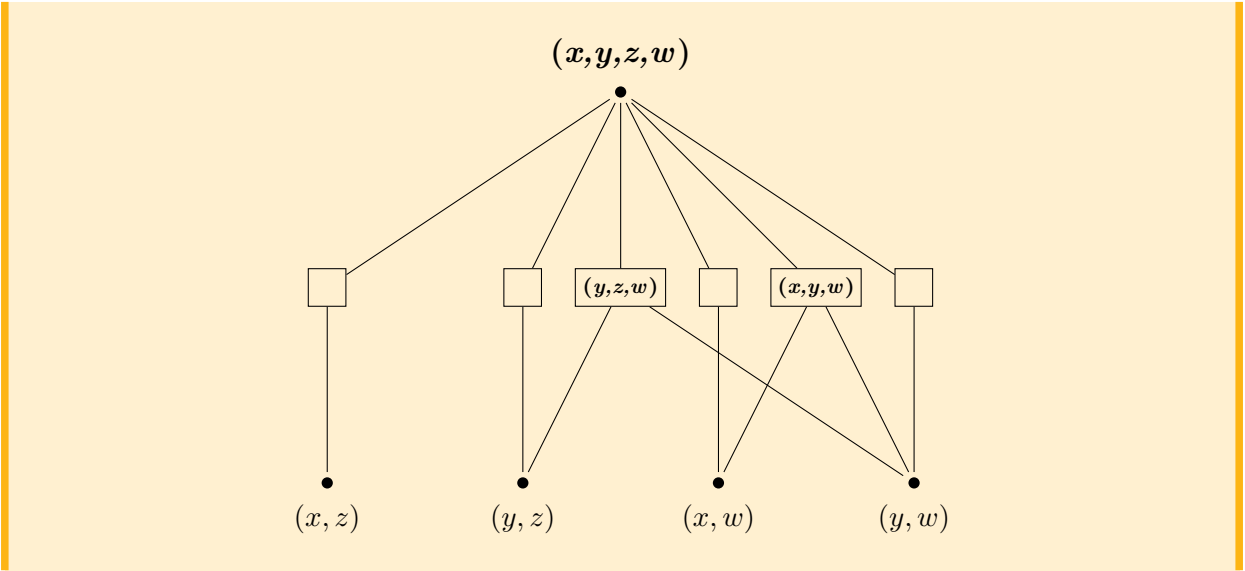
\begin{center}
    \begin{tikzpicture}[scale = 1.3]
        \node[label=above:{\textbf{(\textit{x,y,z,w})}}] (top) at (10, 3) {};
        
        \node[shape=rectangle,draw=black,minimum size=0.5cm] (mid1) at (7, 1) {};
        \node[shape=rectangle,draw=black,minimum size=0.5cm] (mid2) at (9, 1) {};
        \node[shape=rectangle,draw=black,minimum size=0.5cm] (mid3) at (11, 1){};
        \node[shape=rectangle,draw=black,minimum size=0.5cm] (mid4) at (13, 1){};
        
        \node[shape=rectangle,draw=black,minimum width=1.2cm,minimum height=.5cm,label=center:{\scriptsize\textbf{(\textit{y,z,w})}}] (shared2) at (10, 1) {};    
        \node[shape=rectangle,draw=black,minimum width=1.2cm,minimum height=.5cm,label=center:{\scriptsize\textbf{(\textit{x,y,w})}}] (shared3) at (12, 1) {};
        
        \node[label=below:{\small$( x,z )$}] (min1) at (7, -1) {};
        \node[label=below:{\small$( y,z )$}] (min2) at (9, -1) {};
        \node[label=below:{\small$( x,w )$}] (min3) at (11,-1) {};
        \node[label=below:{\small$( y,w )$}] (min4) at (13,-1) {};
        
        \foreach \n in {top,min1,min2,min3,min4}
            \node at (\n)[circle,fill,inner sep=1.5pt]{};
        
        \foreach \n in {shared2,shared3,mid1,mid2,mid3,mid4}
        \path [-] (top) edge node[left] {} (\n);
        
        \path [-] (mid1) edge node[left] {} (min1);
        \path [-] (mid2) edge node[left] {} (min2);
        \path [-] (mid3) edge node[left] {} (min3);
        \path [-] (mid4) edge node[left] {} (min4);
        
        \path [-] (shared2) edge node[left] {} (min2);
        \path [-] (shared2) edge node[left] {} (min4);
        
        \path [-] (shared3) edge node[left] {} (min3);
        \path [-] (shared3) edge node[left] {} (min4);
    \end{tikzpicture}
\end{center}
\end{ex}

\section{The Construction}

We will now construct a precompletion $A$ in a similar manner as in \cite{small09} and \cite{byron} so that $A$ has exactly $m$ minimal prime ideals, call them $q_1,q_2, \ldots ,q_m$, and such that, if $P \in \Spec(T)$ is under $\C_i$, then $P \cap A = q_i$. In addition, we control the construction of the precompletion so that the surjection shown in Theorem \ref{success} is a bijection. In the following remark, we introduce conditions on our ring similar to those in \cite{byron}, with the minor alterations of replacing $\Ass(T)$ with $\Min(T)$ where appropriate and also allowing the characteristic of $T$ to be $p^k$ for some prime $p$. \\

\begin{rmk}\label{conditions}
From now on, let $(T, M)$ be a complete local ring of dimension at least one and $\mathcal{P} = (\C,\{ \C_i \}_{i = 1}^{m})$ be a minfeasible partition such that at least one of the following is true:
\begin{enumerate}[label={(\roman*)}]
    \item $\textup{char}(T) = p$ for some prime $p$.
    \item $\textup{char}(T) = 0$ and $M \cap \mathbb{Z} = (0)$.
    \item $\textup{char}(T) = 0$ and for all $P \in \C$, $M \cap \mathbb{Z} \not \subseteq P$.
    \item $\textup{char}(T) = 0$, $M \cap \mathbb{Z} = (p)$ for some prime $p$, and:
    \begin{enumerate}
        \item For each $P \in \C$ and each $Q \in \Ass(T)$ with $Q \subseteq P$, $p \in Q$ whenever $p \in P$.
        \item For each subcollection $\C_i$ and for any $P,P' \in \C_i$, $p \in P$ if and only if $p \in P'$.
        \item For each $Q \in \Ass(T)$ such that $Q \subseteq P$ for some $P \in \C_i$ such that $\C_i \neq \{Q^*\}$ for some $Q^* \in \Min (T)$, if $p \in Q$, then $\textup{Ann}_T(p) \not \subseteq Q$.
\end{enumerate}
    \item $\textup{char}(T) = p^k$, $k \geq 2$ for some prime $p$ and $\textup{Ann}_T(p) \nsubseteq P$ for all $P \in \C_i$ such that $\C_i \neq \{Q^*\}$ for some $Q^* \in \Min (T)$.
\end{enumerate}
\end{rmk}

We begin with a result from \cite{heitmann94}. It will be used to show that our rings have the desired completion. \\

\begin{prop}[\cite{heitmann94} Proposition 1] \label{cpm}
    If $(A, A \cap M)$ is a quasi-local subring of a complete local ring $(T, M)$, the map $A \rightarrow T/M^2$ is onto, and $IT \cap A = I$ for all finitely generated ideals $I$ of $A$, then $A$ is Noetherian and the natural homomorphism $\widehat{A} \rightarrow T$ is an isomorphism.
\end{prop}

We now begin preliminary work for our main construction. The following lemmas are taken almost exactly from \cite{small09}, except for the small modification $(\ast)$ mentioned in the previous section. The proofs of these lemmas are either identical to the respective versions in \cite{small09}, or the conclusions they imply are equivalent for SIP-subrings and JMT subrings. These lemmas ensure that while we are building up our desired subring, the properties of JMT subrings are preserved at each step. \\

\begin{lem}[\cite{small09} Unioning Lemma 3.1]\label{old 3.1}
Let $(T, M)$ and $\mathcal{P} = (\C, \{ \C_i \}_{i = 1}^m)$ be as in Remark \ref{conditions}. Let $\mathcal{B}$ be a well-ordered indexing set, and let $\{ R_{\beta} \}_{\beta \in \mathcal{B}}$ be a family of JMT subrings such that if $\beta, \gamma \in \mathcal{B}$ and $\beta < \gamma$, then $R_{\beta} \subseteq R_{\gamma}$. Then R = $\bigcup_{\beta \in B} R_{\beta}$ satisfies conditions (i)-(iv) of Definition \ref{Min-T subring}. Moreover, if there exists some $\lambda < |T|$ such that $|R_{\beta}| < \lambda$ for all $\beta$, then $|R| \leq \max \{ \lambda, |\mathcal{B}| \}$ and $R$ is a JMT subring of T. 
\end{lem}

The proof of Lemma \ref{old 3.1} is identical to the proof of Lemma 3.1 in \cite{small09}, substituting JMT subring and minfeasible partition for SIP-subring and feasible partition, respectively. \\

\begin{lem}[\cite{small09} Localization Lemma 3.2]\label{old 3.2}
Let $(T, M)$ and $\mathcal{P} = (\C, \{ \C_i \}_{i = 1}^m)$ be as in Remark \ref{conditions}. Let $R$ be a subring of $T$ satisfying all conditions for a JMT subring, except that it need not be quasi-local. Then $R_{(R \cap M)}$ is a Min-T subring of $T$ with $|R_{(R \cap M)}| = |R|$. Additionally, if $|R| < |T|$, then $R_{(R \cap M)}$ is a JMT subring of $T$.
\end{lem}

The proof of Lemma \ref{old 3.2} is identical to the proof of Lemma 3.2 in \cite{small09}, substituting Min-T subring, JMT subring, and minfeasible partition for IP-subring, SIP-subring, and feasible partition, respectively. \\

\begin{lem}[\cite{small09} Adjoining Lemma 3.3]\label{old 3.3}
Let $R$ be a subring of a complete local ring $T$. Let $P_1$, $P_2$ be prime ideals of $T$ such that $P_1 \cap R = P_2 \cap R$. Suppose that $u + P_i \in T/P_i$ is transcendental over $R/(P_i \cap R)$ for $i = 1, 2$. Then $P_i \cap R[u] = P_2 \cap R[u]$. Furthermore, if $\textup{Ann}_T(p) \nsubseteq P_1$ for all $p \in R \cap P_1$, then $\textup{Ann}_T(p) \nsubseteq P_1$ for all $p \in R[u] \cap P_1$.
\end{lem}

This lemma is exactly the same as Lemma 3.3 from \cite{small09}. Note that the conclusions, which are required conditions for SIP-subrings, are the same conditions we require for JMT subrings, so we will be able to use this lemma in the same way as in \cite{small09}.

\par The following lemma is quoted directly from \cite{aiello}. Along with Lemma \ref{Byron 2.2}, this result will help us find transcendental elements to adjoin to our intermediate subrings. \\

\begin{lem}[\cite{aiello} Lemma 2.7]\label{old 3.4}
Let $(T, M)$ be a complete local ring of dimension at least one, $C$ a countable set of incomparable prime ideals of $T$, and $D$ a subset of $T$ such that $|D| < |T|$. Let $I$ be an ideal of $T$ such that $I \nsubseteq P$ for all $P \in C$. Then $I \nsubseteq \bigcup \{ r + P \mid r \in D, P \in C \}$.
\end{lem}

\begin{lem}[\cite{byron} Lemma 2.2]\label{Byron 2.2}
Let $(T, M)$ and $\mathcal{P} = (\C, \{ \C_i \}_{i = 1}^m)$ be as in Remark \ref{conditions}. Let $R$ be a subring of $T$ such that $|R| < |T|$. Let $J$ be an ideal of $T$ such that $J \nsubseteq P$ for every $P \in \C$. Let $t, q \in T$. Then there exists an element $t' \in J$ such that, for every $P \in \C$ with $q \notin P$, $t + qt' + P \in T/P$ is transcendental over $R/(P \cap R)$. If, in addition, $Q \in \Min(T)$, $P \in \C$ with $Q \subseteq P$, $q \notin P$, and $R \cap P = R \cap Q$, then $t + qt' + Q \in T/Q$ is transcendental over $R/(Q \cap R)$.
\end{lem}

The proof of Lemma 2.2 in \cite{byron} does not rely on the condition that every $Q \in \Ass(T)$ be contained in some $P \in \C$. Additionally, the proof shows the last part of this result for $Q \in \Ass(T)$, $Q \in P$ for some $P \in \C$, which certainly implies the result for $Q \in \Min(T)$. Thus the result holds for minfeasible partitions as well.

\par In their versions of Lemmas \ref{old 3.6}, \ref{Byron 2.8}, \ref{old 3.8}, and Theorem \ref{main theorem}, the authors of \cite{small09} and \cite{byron} ensure that for each ideal $J$ of $T$ satisfying $J \nsubseteq P$ for all $P \in \C$, the final ring contains a nonzerodivisor of $J$. In this way, they guarantee that $J$ is not in the formal fiber of a minimal prime ideal of $A$. In our usage of these lemmas, we will instead find elements $u \in J$ such that $u \notin \bigcup_{Q \in \Min(T)} Q$, which will still guarantee that $J$ is not in the formal fiber of a minimal prime ideal of $A$. This is our only modification to Lemmas \ref{old 3.6}, \ref{Byron 2.8}, \ref{old 3.8}, and the main construction of Theorem \ref{main theorem} as they appear in \cite{small09} and \cite{byron}, and the proofs are equivalent. \\

\begin{lem}\label{old 3.6}
Let $(T,M)$ and $\mathcal{P} = \{ \C, \C_i \}^m_{i = 1}$ be as in Remark \ref{conditions}, and let $J$ be an ideal of $T$ such that $J \nsubseteq P$ for every $P \in \C$. Let $R$ be a JMT subring of $T$ and $t + J \in T/J$. 
Then there exists a JMT subring $S$ of $T$ such that $R \subseteq S \subset T$, $t + J$ is in the image of the map $S \rightarrow T/J$, and $|S| = |R|$. Moreover, if $t \in J$, then $S \cap J$ contains an element $u \in T$ such that $u \notin \bigcup_{Q \in \Min(T)} Q$.
\end{lem}

\begin{proof}
The proof is almost exactly the same as the proof of Corollary 3.6 in \cite{small09}, except for the modification that if $t \in J$, then $S \cap J$ contains an element $u \in T$ such that $u \notin \bigcup_{Q \in \Min(T)} Q$ instead of containing a nonzerodivisor. Use Lemma \ref{Byron 2.2} with $q = 1$. Then $q \notin P$ for every $P \in \Spec(T)$, so it is possible to choose $t' \in J$ such that $t + t' + P \in T/P$ is transcendental over $R/(P \cap R)$ for every $P \in \C \cup \Min(T)$. Consider the ring $S = R[t + t']_{(M \cap R[t + t'])}$. By Lemma \ref{old 3.3}, $R[t + t']$ satisfies conditions (ii), (iii), and (iv) of being a JMT subring. 
Furthermore, $|R[t + t']| = |R|$. By Lemma \ref{old 3.2}, $S$ is a JMT subring of $T$, and $|S| = |R|$. Moreover, $(t + t') \in S$ and $(t + t') + J = t + J$, so $t + J$ is in the image of the map $S \rightarrow T/J$.
\par Suppose $t \in J$ and $t + t' \in \bigcup_{Q \in \Min(T)} Q$. Then $t + t' \in Q$ for some $Q \in \Min(T)$. However, $Q \subseteq P$ for some $P \in \C$, and so $(t + t') + P = 0 + P$. Hence, $t + t' + P \in T/P$ is algebraic over $R/(P \cap R)$, a contradiction. Thus, $t + t'$ is an element of $T$ not in $\bigcup_{Q \in \Min(T)} Q$ but contained in $S \cap J$.
\end{proof}

The following is a preliminary lemma that will be used in the proof of Lemma \ref{Byron 2.8}. It is taken directly from \cite{byron}. \\

\begin{lem}[\cite{byron} Lemma 2.7]\label{Byron 2.7}
Let $R$ be a ring and let $P$ be a prime ideal of $R$. Let $a, b \in R$, and suppose $a \notin P$. If $l$ and $l'$ are units such that $b + la \in P$ and $b + l'a \in P$ then $l + P = l' + P$.
\end{lem}

\par Lemma \ref{Byron 2.8} is the main step in our construction. Using this result, we will be able to successively build up JMT subrings.  Note that our alteration to condition (iv) of Definition \ref{Min-T subring} means that there can exist minimal prime ideals $Q^*$ under $\C_i = \{ Q^* \}$ for some $i$ with $a \in R \cap Q^*$ and $\Ann_T(a) \subseteq Q^*$, but note that this condition is only used in \cite{byron} in order to ensure that $P \cap R = P' \cap R$ for all $P,P'$ under $\mathcal{C}_i$ - a condition which is trivially satisfied in the case $\C_i = \{ Q^* \}$. 
The proof of Lemma \ref{Byron 2.8} follows closely along with that of Lemma 2.8 in \cite{byron}, and we omit parts that are verbatim from \cite{byron}.\\

\begin{lem} \label{Byron 2.8}
Let $(T, M)$ and $\mathcal{P} = (\C, \{ \C_i \}_{i = 1}^m)$, be as in Remark \ref{conditions}. Let $R$ be a JMT subring of $T$. Then, for any finitely generated ideal $I$ of $R$ and any $c \in IT \cap R$, there exists a subring $S$ of $T$ with the following properties:
\begin{enumerate}[label={(\roman*)}]
    \item $R \subseteq S$
    \item $S$ is a JMT subring of $T$
    \item $|S| = |R|$
    \item $c \in IS$
\end{enumerate}
\end{lem}

\begin{proof}
We proceed by induction on the number of generators of $I$. The first part of the proof covers the case where $I$ is a principal ideal (base case), and is verbatim from the proof of Lemma 3.7 in \cite{small09} and Lemma 2.8 in \cite{byron}.
Since those proofs still work after substituting minfeasible partitions and JMT subrings for their analogs, we will not restate this step.
Note that although if $\C_i = \{Q^*\}$, for $Q^* \in \Min (T)$ it is possible that $a \in Q^*$ and $\Ann_T (a) \subseteq Q^*$, we do not need to consider these ideals because there are no other prime ideals under $\C_i$, and so all the intersection properties of JMT subrings follow trivially for $Q^*$.
\par Before we proceed, note that for any $r \in R$, it follows from condition (iii) of Definition \ref{Min-T subring} that given $\C_j$ in the minfeasible partition, $r$ must be in all $P \in \C_j$ or no $P \in \C_j$. Given this, we introduce the abuse of notation $r \in \C_j$ to mean that $r \in P$ for all $P \in \C_j$.
\par The next part of the proof is the inductive step, where $I = (a_1, \ldots, a_k)R$ for $k \ge 2$. This step is split into two parts. We first show that the lemma follows in the case where
\[
\{\C_i \mid a_1 \in \C_i \} = \{ \C_j \mid a_2 \in \C_j \}. \tag{$\ast  \ast$}
\]
Then, we show that it is always possible to find a generating set such that $(\ast \ast)$ holds. The proof for the case where $(\ast \ast)$ holds is the same as the proof of Lemma 2.8 in \cite{byron}, so we will not restate it here.  We now move to the second part of the proof.  That is, we assume the lemma holds for the case $(\ast \ast)$ and we show that, in all cases, we can reduce to the case where $(\ast \ast)$ holds.  
Similarly to the principal case, note that although for $\C_i = \{Q^*\}$ for $Q^* \in \Min (T)$ it is possible that $a \in Q^*$ and $\Ann_T (a) \subseteq Q^*$, we do not need to consider these ideals because there are no other prime ideals in $\C_i$, and so all the intersection properties of JMT subrings follow trivially for $Q^*$.
\par The approach remains the same as in \cite{byron}. Use Lemma \ref{old 3.6} with $J = M$ and $ t =0$ to find a JMT subring $R_0$ of $T$ such that $R \subseteq R_0 \subset T$, $|R_0| = |R|$, and $R_0 \cap M$ contains an element $m_0 \notin \bigcup_{Q \in \Min(T)} Q$. 
By condition (ii) of the definition of JMT subring, if $P \in \C$, then $R_0 \cap P$ contains only elements that are contained in some $Q \in \Min(T)$, so $m_0 \notin \C_i \cap R_0$ for all $i$. Now, for each $P \in \C$, let $D_{(P)}$ be a full set of coset representatives of the cosets $t + P \in T/P$ that are algebraic over $R_0/(R_0 \cap P)$. 
Let $D' = \bigcup_{P \in \C} D_{(P)}$. Use Lemma \ref{old 3.4} (recall $\C$ can be countably infinite) with $I = M$ and $D = D' \cup \{ m_0\}$ to find an element $m_1$ of M such that, for all $P \in \C$, $m_1 + P \neq m_0 + P$ and $m_1 + P \in T/P$ is transcendental over $R_0/(R_0 \cap P)$. 
If $Q \in \Min(T)$ and $P \in \C$ with $Q \subseteq P$ then $P \cap R_0 = Q \cap R_0$. It follows that $m_1 + Q \in T/Q$ is transcendental over $R_0/(R_0 \cap Q)$. Let $R_1 = R_0[m_1]_{(R_0 \cap M)}$. By Lemmas \ref{old 3.2} and \ref{old 3.3}, $R_1$ satisfies conditions (ii), (iii), and (iv) of JMT subrings. Clearly $R_1$ is infinite, and $|R_1| = |R_0| < |T|$. Thus $R_1$ is a JMT subring of $T$.
\par Now, repeat the above procedure with $R_0$ replaced by $R_1$ and $D$ replaced by $D' \cup \{ m_0, m_1 \}$ to obtain an element $m_2$ of $M$ and a JMT subring $R_2$ of T such that $R_1 \subseteq R_2$ and, for every $P \in \C$, we have $m_2 + P \neq m_0 + P$ and $m_2 + P \neq m_1 + P$. 
Continue so that for every $n \in \{ 1, 2, \ldots \}$, we find $m_n \in M$ and $R_n$ such that $R_{n - 1} \subseteq R_n$, $|R_n| = |R_{n - 1}|$, $R_n$ is a JMT subring of T, and for every $P \in \C$ and every $i < n$, $m_i + P \neq m_n + P$. Let $R' = \bigcup_{i = 1}^{\infty} R_i$. Then if $P \in \C$, we have $m_i + P = m_j + P$ if and only if $i = j$. In addition, by Lemma \ref{old 3.1}, $R'$ is a JMT subring and $|R'|=|R|$. Since $m_0 \in R' \cap M$, and $m_0 \notin Q$ for all $Q \in \Min(T)$, we have $\C_i \cap R' \neq M \cap R'$ for all $i = 1, 2, \ldots, m$. Also note that in the ring $R'/(\C_i \cap R')$, we have $m_k + (\C_i \cap R') = m_j + (\C_i \cap R')$ if and only if $k = j$. It follows that $(1+m_k) + (\C_i \cap R') = (1+m_j) + (\C_i \cap R')$ if and only if $k = j$.
\par Since $m_0$ is a nonunit, $m_0 + 1$ is a unit. Consider an ideal of $R'$ of the form $(m_0 a_1 + ua_2, a_1 - ua_2, a_3, \ldots, a_k)$, where $u \in R'$ is a unit that will be chosen so that $(\ast \ast)$ holds. This ideal is equal to $((m_0 + 1)a_1, (m_0 + 1)ua_2, a_3, \ldots, a_k)R'$, and is therefore equal to $IR'$.
Since neither $m_0$ nor $u$ are in any $\C_i \cap R'$, we have that for each $\C_i$, $m_0 a_1 \in \C_i \cap R'$ if and only if $a_1 \in \C_i \cap R'$ and $ua_2 \in \C_i \cap R'$ if and only if $a_2 \in \C_i \cap R'$. It follows that if $a_1, a_2 \in \C_i \cap R'$, then $m_0 a_1 + ua_2, a_1 - ua_2 \in \C_i \cap R'$. On the other hand, if $a_1 \in \C_i \cap R'$ and $a_2 \notin \C_i \cap R'$ or $a_1 \notin \C_i \cap R'$ and $a_2 \in \C_i \cap R'$, then $m_0 a_1 + ua_2, a_1 - ua_2 \notin \C_i \cap R'$.
\par Finally, consider the case where $a_1,a_2 \notin \C_i \cap R'$. As $a_2 \notin \C_i \cap R'$, by Lemma \ref{Byron 2.7} every unit $l \in R'$ such that $m_0 a_1 + la_2 \in \C_i \cap R'$ is in the same coset of $R'/(\C_i \cap R')$. Similarly, every unit $l' \in R'$ such that $a_1 - l'a_2 \in \C_i \cap R'$ is in the same coset of $R'/(\C_i \cap R')$. For each i, let $l_{i_{+}}$ be a representative of the coset of $R'/(\C_i \cap R')$ containing all units $l$ such that $m_0 a_1 + la_2 \in \C_i \cap R'$ and let $l_{i_{-}}$ be a representative of the coset of $R'/(\C_i \cap R')$ containing all units $l'$ such that $a_1 - l'a_2 \in \C_i \cap R'$. Since $\mathcal{L} = \bigcup_{i = 1}^{m} \{ l_{i_{+}},l_{i_{-}} \}$ is a finite set of elements of $R'$ and $\mathcal{G} = \{ \C_i \cap R' \}_{i = 1}^m$ is a finite set of prime ideals of $R'$, the set $\{ l + P \mid l \in \mathcal{L}, P \in \mathcal{G} \}$ is a finite set. Suppose for some $r \neq k$, $l \in \mathcal{L}$, and $P \in \mathcal{G}$, we have $m_r + 1 \in l + P$ and $m_k + 1 \in l + P$. Then $m_r + P = m_k + P$, a contradiction. As the set $\{ m_i + 1 \}_{i = 1}^{\infty}$ is infinite, there must be a positive integer $r$ such that $m_r + 1 \notin l + P$ for all $l \in \mathcal{L}$ and all $P \in \mathcal{G}$. So there exists a unit $u = m_r + 1 \in R'$ such that, for all $i$, we have $m_0 a_1 + ua_2, a_1 - ua_2 \notin \C_i \cap R'$. Thus $(\ast \ast)$ holds. Replacing $R$ with $R'$ and $I$ with $IR'$, and using the fact that the lemma holds in the case $(\ast \ast)$, we can find a JMT subring $S$ of $T$ containing $R'$ so that $c \in (IR')S = IS$.
\end{proof}

\begin{lem} \label{old 3.10}
Let $(T, M)$ and $\P = (\C, \{ \C_i \}^m_{i = 1})$ be as in Remark \ref{conditions} and fix $\C_i$. Let $R$ be a semi-JMT subring of $T$ and let $p_i \in T$ be given such that $p_i \in Q$ for every $Q \in \Min(T)$ under $\C_i$, and $p_i \not \in P$ for every $P \in \C_j$, where $j \neq i$. Suppose further that if $\C_i \neq \{Q^*\}$ for some $Q^* \in \Min (T)$, then $\textup{Ann}_T(p_i) \not \subseteq P$ for any $P \in \C_i$. Then there exists a unit $u \in T$ such that $R[up_i]_{(R[up_i]\cap M})$ is a semi-JMT subring of $T$.
\end{lem}

\begin{proof}
If $\C_i \neq \{ Q^* \}$ for some $Q^* \in \Min (T)$, then the unit $u$ constructed in the proof of Lemma 3.10 of \cite{small09} suffices. The proof in \cite{small09} does not rely on every $Q \in \Ass(T)$ being contained in some $P \in \C$, and the conditions it implies for semi-SIP-subrings are the same as those for semi-JMT subrings, so it holds for minfeasible partitions and semi-JMT subrings as well.
\par If $\C_i = \{ Q^* \}$ for some $Q^* \in \Min (T)$, then properties (ii)-(iv) of Definition \ref{Min-T subring} follow trivially for $\C_i$, and the proof from \cite{small09} implies all other properties of semi-JMT subrings.
\end{proof}

\begin{lem} \label{old 3.8}
Let $(T, M)$ and $\P = ( \C, \{ \C_i \}^m_{i = 1})$ be as in Remark \ref{conditions}. Let $J$ be an ideal of $T$ such that $J \nsubseteq P$ for all $P \in \C$, and let $u + J \in T/J$. Suppose $R$ is a JMT subring. Then there exists a JMT subring $S$ of $T$ such that
\begin{enumerate}[label={(\roman*)}]
    \item $R \subseteq S \subset T$;
    \item if $u \in J$, then $S \cap J$ contains an element $t \notin \bigcup_{Q \in \Min(T)} Q$;
    \item $u + J$ is in the image of the map $S \rightarrow T/J$;
    \item for every finitely generated ideal $I$ of $S$, we have $IT \cap S = I$;
    \item $|R| = |S|$
\end{enumerate}
\end{lem}

\begin{proof}
This proof follows almost exactly the proof of Lemma 3.8 in \cite{small09}. First, we use Lemma \ref{old 3.6} (in this paper) to find a JMT subring $R'$ such that $R \subseteq R'$, $u + J$ is in the image of $R' \rightarrow T/J$, if $u \in J$ then $J \cap R'$ contains an element $t \notin \bigcup_{Q \in \Min(T)} Q$, and $|R'| = |R|$. We will construct a ring $S$ such that $R' \subseteq S \subset T$ and so conditions (i)-(iii) of the lemma hold for $S$. 
\par Let $\Omega = \{ (I, c) \mid I \text{ finitely generated, } c \in IT \cap R' \}$. The cardinality of the set of finitely generated ideals of $R'$ is less than or equal to $|R'|$. Hence, $|\Omega| = |R'| < |T|$. Well-order $\Omega$ so that it has no maximal element, and let 0 denote the minimal element of $\Omega$. For each $\alpha \in \Omega$, we define $\gamma(\alpha) = \sup \{ \beta \in \Omega \mid \beta < \alpha \}$. Let $R_0 = R'$.
\par Let $\lambda \in \Omega$. Assume that $R_\beta$ has been defined for all $\beta < \lambda$, such that $R_\beta$ is a JMT subring, and $|R_\beta| = |R'|$. Suppose $\gamma(\lambda) < \lambda$, and let $(I, c) = \gamma(\lambda)$. Then, using Lemma \ref{Byron 2.8}, we construct $R_\lambda$ such that $R_{\gamma(\lambda)} \subseteq R_\lambda$ and $c \in IR_\lambda$. Note that $|R_\lambda| = |R_{\gamma(\lambda)}| = |R'|$.
\par Next, suppose $\gamma(\lambda) = \lambda$. In this case, we define $R_\lambda = \cup_{\beta < \lambda} R_\beta$. Since $|R_\beta| = |R'|$ for all $\beta < \lambda$, and $|\Omega| = |R'|$, Lemma \ref{old 3.1} implies that $R_\lambda$ is a JMT subring, and $|R_\lambda| = |R'|$.
\par Define
$$
R_1 = \bigcup_{\alpha \in \Omega} R_\alpha.
$$
Then $|R_1| = |R'|$, and $R_1$ is a JMT subring of $T$. If $I$ is a finitely generated ideal of $R_0$, and $c \in IT \cap R_0$, then $(I, c) = \gamma(\alpha)$ for some $\alpha$ such that $\gamma(\alpha) < \alpha$. It follows that $c \in IR_\alpha \subseteq IR_1$. Hence, $IT \cap R_0 \subseteq IR_1$ for every finitely generated ideal $I$ of $R_0$.
\par We repeat this process for $R_1$, and obtain a JMT subring $R_2$ containing $R_1$ such that $IT \cap R_1 \subseteq IR_2$ for each finitely generated ideal $I$ of $R_1$, and $|R_2| = |R'|$. Continuing in this fashion, we construct a chain of JMT subrings $R_0 \subseteq R_1 \subseteq R_2 \subseteq \cdots$ such that $IT \cap R_n \subseteq IR_{n + 1}$ for every finitely generated ideal $I$ of $R_n$ and $|R_n| = |R'|$ for all $n \in \mathbb{N}$. 
\par Let
$$
S = \bigcup_{i = 1}^\infty R_i.
$$
Then, $|S| = |R'| = |R|$, and $S$ is a JMT subring of $T$. Let $I = (s_1, \ldots, s_n)$ be an ideal of $S$, and let $c \in IT \cap S$. Then $c = s_1 t_1 + \cdots + s_n t_n$, where $t_i \in T$, and where each $s_k \in R_{m_k}$ for some $m_k \in \mathbb{N}$. Now, there exists an $m_0 \in \mathbb{N}$ such that $c \in R_{m_0}$. Let $N = \max\{ m_k \mid 0 \le k \le n \}$. Then $c \in (s_1, \ldots, s_n)T \cap R_N \subseteq (s_1, \ldots, s_n)R_{N + 1} \subseteq IS$. Therefore, $IT \cap S = I$. It follows that $S$ is our desired JMT subring of $T$. 
\end{proof}

We now need to construct a base JMT subring of $T$ on which to build up subsequent JMT subrings. We start with a semi-JMT subring. In \cite{small09}, Arnosti et al.~assume $T$ contains the rationals, so the use of $\mathbb{Q}$ suffices for them. In \cite{byron}, Loepp and Perpetua do not assume that $T$ necessarily contains the rationals, so they construct a base subring for every possible case stated in their conditions, as we will now do here.  Parts of the proof of Lemma \ref{Byron 2.12} closely follow the relevant parts of the proof of Lemma 2.12 in \cite{byron} \\

\begin{lem} \label{Byron 2.12}
Let $(T, M)$ and $\P = (\C,\{ \C_i \}_{i = 1}^m)$ be as in Remark \ref{conditions}. Then there exists a semi-JMT subring of T.
\end{lem}

\begin{proof}
\textbf{Characteristic 0:} Suppose the characteristic of $T$ is zero. We prove the case of each condition in Remark \ref{conditions} separately.
\par \textit{Assume only condition (ii) of Remark \ref{conditions} holds}. That is $M \cap \mathbb{Z} = ( 0 )$. Then $M$ contains no integers, so every integer is a unit, hence $T$ contains the rationals. Thus, $R_0 = \mathbb{Q}$ satisfies the conditions of a semi-JMT subring.
\par \textit{Assume only condition (iii) of Remark \ref{conditions} holds.} Then $M \cap \mathbb{Z} = p \mathbb{Z}$ for some prime integer $p$, but $p$ is not in any $P \in \C$. Choose any $P \in \C$. Since $p \notin P$, we have $P \cap \mathbb{Z} \neq p \mathbb{Z}$, and as $P \cap \mathbb{Z} \subseteq M \cap \mathbb{Z}$, we have $P \cap \mathbb{Z} = ( 0 )$. 
Let $R_0 = \mathbb{Z}_{(p \mathbb{Z})}$, so that $R_0$ is a local subring of $T$. Then $P \cap R_0 = ( 0 )$. Furthermore, for any $Q \in \Min(T)$, $Q \subseteq P$ for some $P \in \C$, so $Q \cap R_0 = ( 0 )$. Conditions (ii)-(iv) of semi-JMT subrings follow easily. Then note that $R_0$ is countably infinite and $|T| \geq c$, so $R_0$ is a semi-JMT subring of T.
\par \textit{Assume only condition (iv) of Remark \ref{conditions} holds.} Then $M \cap \mathbb{Z} = p \mathbb{Z}$. Let $R_0 = \mathbb{Z}_{(p \mathbb{Z})}$, so that $R_0$ is an infinite local subring of $T$ with $|R_0| < |T|$. For any prime ideal $P$ of $T$, $P \cap \mathbb{Z}$ is either $p \mathbb{Z}$, if $p \in P$, or $( 0 )$, if $p \notin P$. Conditions (ii) and (iii) of semi-JMT subrings follow respectively from conditions (iv.a) and (iv.b) of Remark \ref{conditions}, so all we have left is to show that condition (iv.c) of Remark \ref{conditions} implies condition (iv) of Definition \ref{Min-T subring}.
\par For the sake of contradiction, suppose that condition (iv) of the definition of semi-JMT subrings does not hold, so that there exists some $P \in \C_i$ with $\C_i \neq \{Q^*\}$ where $Q^* \in \Min (T)$ and $r \in P \cap R_0$ with $\Ann_T(r) \subseteq P$. If $p \notin P$, then $P \cap R_0 = ( 0 )$ and clearly $\Ann_T(0) \nsubseteq P$, so $p \in P$. 
As $P \cap \mathbb{Z} = p \mathbb{Z}$ and $\Ann_T(p) \subseteq \Ann_T(kp)$ for any integer $k$, we may assume that $r = p$ and $\Ann_T(p) \subseteq P$. 
The set of zerodivisors of $T_P$ is equal to $\bigcup \{ QT_P \mid Q \in \Ass(T), Q \subseteq P \}$. Since $\Ann_{T_P}(p)$ consists entirely of zerodivisors, by the Prime Avoidance Theorem we have $\Ann_{T_P}(p) \subseteq QT_P$ for some $Q \in \Ass(T)$ with $Q \subseteq P$. 
Choose any $a \in \Ann_T(p)$, so that $ap = 0$. In $T_P$, we have $\frac{a}{1} \cdot \frac{p}{1} = \frac{0}{1}$. Thus $\frac{a}{1} \in \Ann_{T_P}(p)$, and by the above argument $\frac{a}{1} \in QT_P$. It follows that $a \in Q$, so that $\Ann_T(p) \subseteq Q$. By condition (iv.c) of Remark \ref{conditions}, $p \notin Q$, but by condition (iv.a) of Remark \ref{conditions}, $p \in Q$ because $p \in P$, a contradiction. Therefore, condition (iv) of Definition \ref{Min-T subring} holds, and $R_0$ is a semi-JMT subring of $T$.
\par \textbf{Characteristic $\pmb{p}$/$\pmb{p^k}$:} Assume that the characteristic of $T$ is some prime integer $p$. Let $R_0 = \mathbb{Z}_p$. Since $R_0$ is a field, no prime ideal of $T$ contains a nonzero element of $R_0$. Therefore,
$$
P \cap R_0 = ( 0 ) = P' \cap R_0 = Q \cap R_0
$$
for any $P,P' \in \C$, $Q \in \Min(T)$. Furthermore, it is trivially true that $\Ann_T(r) \nsubseteq P$ for any $P \in \C$ and $r \in P \cap R_0$. 

\par Assume that the characteristic of $T$ is $p^k$, $k \geq 2$ for some prime integer $p$. Let $R_0 = \mathbb{Z}_{p^k}$. Every element of $R_0$ that is relatively prime to $p$ is a unit, and every element of $R_0$ that is in $( p )$ is in every prime ideal of $T$, because $p$ is nilpotent. Therefore,
\[
P \cap R_0 = ( p ) = P' \cap R_0 = Q \cap R_0
\]
for any $P,P' \in \C$, $Q \in \Min(T)$. By assumption, $\Ann_T(p) \nsubseteq P$ for any $P \in \C_i$ with $\C_i \neq \{Q^*\}$ where $Q^* \in \Min (T)$.  If $r \in P \cap R_0$ then $p$ divides $r$, so $\Ann_T(p) \nsubseteq P$ implies that $\Ann_T(r) \nsubseteq P$. 

In either case we obtain a finite subring that satisfies conditions (ii)-(iv) of Definition \ref{Min-T subring}. By adjoining an element to $R_0$, we will make $R_0$ infinite, and thus create a semi-JMT subring:


\par Use Lemma \ref{Byron 2.2} with $J=M$, $t=0$, and $q=1$ to find an element $t' \in M$ such that for every $P \in \C$, $t' + P \in T/P$ is transcendental over $R_0/(P \cap R_0)$. Then $R_0[t']$ satisfies properties (ii)-(iv) of Definition \ref{Min-T subring} by Lemma \ref{old 3.3} and is countably infinite, so it satisfies properties (i) and (v) as well. Then $R_0[t']_{(R_0[t'] \cap M)}$ is the desired semi-JMT subring.
\end{proof}

\begin{thm} \label{Byron 2.13}
Let $(T, M)$ be a complete local ring and $\P = (\C, \{ \C_i \}_{i = 1}^m)$ be a minfeasible partition. If all the conditions of Remark \ref{conditions} hold, then there exists a JMT subring of $T$. On the other hand, if none of the conditions in Remark \ref{conditions} hold, and the following hold:
\begin{enumerate}[label={(\roman*)}]
\item if $\textup{char}(T) = 0$, $M \cap \mathbb{Z} = (p)$, and $T$, $\P$ satisfy conditions (iv.a) and (iv.b) of Remark \ref{conditions}, then for each $Q \in \Ass(T)$ such that $Q \subseteq P$ for some $P \in \C_i$ such that $\C_i \neq \{Q^*\}$ for some $Q^* \in \Min (T)$, if $p \in Q$, then $\textup{Ann}_T(p) \nsubseteq Q$.
\item if $\textup{char}(T) = p^k$, $k \ge 2$ for some prime $p$, then $\textup{Ann}_T(p) \nsubseteq P$ for all $P \in \C_i$ such that $\C_i \neq \{Q^*\}$ for some $Q^* \in \Min (T)$.
\end{enumerate}
then there does not exist a subring of $T$ that satisfies conditions (ii) and (iii) of Definition \ref{Min-T subring}.
\end{thm}

\begin{proof}
Since the conditions of Remark \ref{conditions} hold, there exists a semi-JMT subring of $T$ by Lemma \ref{Byron 2.12}. We construct a JMT subring of $T$ in the exact same way as in the proof of Lemma 3.11 in \cite{small09}, replacing $\mathbb{Q}$ in that proof with the semi-JMT subring of Lemma \ref{Byron 2.12}.
\par If the characteristic of $T$ is not zero, then we will show that the characteristic must be $p^k$ for some prime $p$. Suppose otherwise. Then $T$ has nonzero characteristic $n$ where $n = ab$ for some relatively prime integers $1 < a$, $b < n$. Then $a$ and $b$ are zerodivisors, and consequently nonunits, so $a, b \in M$. Then $( a,b ) \subseteq M$. By Bezout's identity, there exist integers $r$ and $s$ such that $ra + sb = 1$, implying $1 \in M$, a contradiction. Thus $\textup{char}(T) = 0$ or $\textup{char}(T) = p^k$ for some prime $p$.
\par We now show that, assuming conditions (i) and (ii) in the statement of the theorem hold, at least one of the conditions of Remark \ref{conditions} necessarily holds for there to exist a subring of $T$ satisfying conditions (ii) and (iii) of Definition \ref{Min-T subring}. We will assume that such a subring exists when conditions (i)-(iii) and (v) of Remark \ref{conditions} fail and show that condition (iv) of Remark \ref{conditions} must hold. Note that if (i)-(iii) and (v) of Remark \ref{conditions} fail, then $\textup{char}(T) = 0$, $M \cap \mathbb{Z} = ( p )$ for some prime integer $p$, and there exists $P \in \C$ such that $M \cap \mathbb{Z} \subseteq P$. So we must be in case (iv). Note that if (iv.a) and (iv.b) hold, then (iv.c) holds by the assumption in the statement of the theorem. We will now show that (iv.a) and (iv.b) hold.
\par Suppose that there exists a subring of $T$ satisfying conditions (ii) and (iii) of Definition \ref{Min-T subring} when conditions (i)-(iii), (v) and (iv.a) of Remark \ref{conditions} fail. Then $\textup{char}(T) = 0$, $M \cap \mathbb{Z} = p \mathbb{Z}$ for some prime $p$, and there exists some $P \in \C$ and some $Q \in \Ass(T)$ contained in $P$ such that $p \in P$ but $p \notin Q$. By condition (ii) of Definition \ref{Min-T subring}, $P \cap R = Q \cap R$, so $p \notin R$. But this is impossible, as any subring of $T$ must contain the integers. 
\par Now suppose that there exists a subring of $T$ satisfying conditions (ii) and (iii) of Definition \ref{Min-T subring} when conditions (i)-(iii), (v) and (iv.b) of Remark \ref{conditions} fail. Then there exists an $i$ such that $P \cap \mathbb{Z} \neq P' \cap \mathbb{Z}$ for some $P,P' \in C_i$. As any subring of $T$ contains the integers, $P \cap R \neq P' \cap R$, contradicting condition (iii) of Definition \ref{Min-T subring}.
\end{proof}

We now prove our main theorem of this section, a generalization of Theorem 2.14 from \cite{byron}. Note that, although the conditions on the complete local ring $T$ are complicated, essentially the theorem says that for most complete local rings $T$ (particularly for all ``nice" cases) we can construct a precompletion $A$ of $T$ such that we control the formal fibers of the minimal prime ideals of $A$ very precisely. In addition, this construction gives precise control of the prime ideals of $A$ of positive height that contain particular minimal prime ideals of $A$. In particular, this construction achieves the upper bound on the set $\S_A(X)$ of Theorem \ref{success} so that the surjection from $\S_T(X)$ onto $\S_A(X)$ is a bijection. \\

\begin{thm} \label{main theorem}
Let $(T, M)$ be a complete local ring of dimension at least one and $\P = (\C, \{ \C_{i = 1}^m \})$ be a minfeasible partition and assume that:
\begin{enumerate}[label={(\roman*)}]
\item If $\textup{char}(T) = 0$, $M \cap \mathbb{Z} = (p)$, and $T, \P$ satisfy conditions (iv.a) and (iv.b) of Remark \ref{conditions}, then for each $Q \in \Ass(T)$ such that $Q \subseteq P$ for some $P \in \C_i$ such that $\C_i \neq \{Q^*\}$ for some $Q^* \in \Min(T)$, if $p \in Q$, then $\textup{Ann}_T(p) \not \subseteq Q$.
\item If $\textup{char}(T) = p^k$, $k \ge 2$ for some prime $p$, then $\textup{Ann}_T(p) \nsubseteq P$ for all $P \in \C_i$ such that $\C_i \neq \{Q^*\}$ for some $Q^* \in \Min (T)$.
\end{enumerate}
Then $T$ is the completion of a local subring $A$ such that:
\begin{enumerate}[label={(\roman*)}]
    \item $\Min(A) = \{ \C_1 \cap A, \ldots, \C_m \cap A \}$, and the formal fiber of $\C_i \cap A$ has countably many maximal elements, which are precisely the elements of $\C_i.$
    \item If $J$ is an ideal of $T$ such that $J \subsetneq P$ for every $P \in \C$, then the natural map $A \rightarrow T/J$ is onto so that there is a bijection between $\{ J \in \Spec(T) \mid J \not \subseteq P \text{ for every } P \in \C \}$ and $\{ j \in \Spec(A) \mid j \not \subseteq q \text{ for every } q \in \Min(A) \}$.
    \item Using the notation of Theorem \ref{success}, for all $X \subseteq \{ 1, \ldots, m \}$ we have a bijection between $\mathcal{S}_T(X)$ and $\mathcal{S}_A(X)$.
\end{enumerate}
if and only if one of the following is true:
\begin{enumerate}[label={(\roman*)}]
    \item $\textup{char}(T) = p$ for some prime $p$.
    \item $\textup{char}(T) = 0$ and $M \cap \mathbb{Z} = (0)$.
    \item $\textup{char}(T) = 0$ and for all $P \in \C$, $M \cap \mathbb{Z} \not \subseteq P$.
    \item $\textup{char}(T) = 0$, $M \cap \mathbb{Z} = (p)$ for some prime p, and:
    \begin{enumerate}
        \item For each $P \in \C$ and each $Q \in \Ass(T)$ with $Q \subseteq P$, $p \in Q$ whenever $p \in P$.
        \item For each subcollection $\C_i$ and for any $P,P' \in \C_i$, $p \in P$ if and only if $p \in P'$.
        \item For each $Q \in \Ass(T)$ such that $Q \subseteq P$ for some $P \in \C_i$ such that $\C_i \neq \{Q^*\}$ for some $Q^* \in \Min(T)$, if $p \in Q$, then $\textup{Ann}_T(p) \not \subseteq Q$.
\end{enumerate}
    \item $\textup{char}(T) = p^k$, $k \ge 2$ for some prime $p$ and $\textup{Ann}_T(p) \nsubseteq P$ for all $P \in \C_i$ such that $\C_i \neq \{Q^*\}$ for some $Q^* \in \Min (T)$.
\end{enumerate}
\end{thm}

\begin{proof}
In Theorem \ref{Byron 2.13}, we proved that given the first two assumptions, these conditions are necessary for there to exist a subring that satisfies conditions (ii) and (iii) of Definition \ref{Min-T subring}. Such a subring is necessary to control the intersection of the prime ideals under elements in the minfeasible partition, and thus necessary for us to achieve control over the minimal prime ideals of the precompletion and their formal fibers as desired.
\par The proof of existence of such a subring is almost exactly the same as the proof of Theorem 2.14 in \cite{small09}. Let $\Omega = \{ u + J \mid u \in T, J \nsubseteq P \text{ for all } P \in \C \}$ equipped with a well-ordering $<$, such that every element has strictly fewer than $|\Omega|$ predecessors. Note that
\[
|\{ J \mid J \text{ is an ideal of } T \text{ with } J \not \subseteq P \text{ for every } P \in \C \}| \leq |T|.
\]
For each $\alpha \in \Omega$, we let $|\alpha| = |\{ \beta \in \Omega \mid \beta \leq \alpha \}|$. Let 0 denote the first element of $\Omega$, and let $R_0$ be the JMT subring constructed in Lemma \ref{Byron 2.13}. For each $\lambda \in \Omega$ after the first, we define $R_{\lambda}$ recursively as follows: assume $R_{\beta}$ is defined for all $\beta < \lambda$ such that $R_{\beta}$ is a JMT subring, and $|R_{\beta}| \leq |\beta||R_0|$ for all $\beta < \alpha$. Let $\gamma(\lambda) = u + J$ denote the least upper bound of the set of predecessors of $\lambda$. If $\gamma(\lambda) < \lambda$, we use Lemma \ref{old 3.8} with $R = R_{\gamma(\lambda)}$ to find a JMT subring $R_{\lambda}$ such that
\begin{enumerate}[label={(\roman*)}]
    \item $R_{\gamma(\lambda)} \subseteq R_{\lambda} \subseteq T$;
    \item if $u \in J$, then $R_{\lambda} \cap J$ contains an element $t \notin \bigcup_{Q \in \Min(T)} Q$;
    \item the coset $\gamma(\lambda) = u+J$ is in the image of the map $R_{\lambda} \rightarrow T/J$;
    \item for every finitely generated ideal $I$ of $R_{\lambda}$, we have $IT \cap R_{\lambda} = I$;
\end{enumerate}
In this case,
\[
|R_{\lambda}| = |R_{\gamma(\lambda)}| \leq |\gamma(\lambda)||R_0| \leq |\lambda||R_0|
\]
On the other hand, if $\gamma(\lambda) = \lambda$, we let
\[
R_{\lambda} = \bigcup_{\beta < \lambda} R_\beta.
\]
Then $|\lambda| < |\Omega| \leq |T|$, and $|R_{\lambda}| \leq |\lambda||R_0|$. By Lemma \ref{old 3.1}, $R_{\lambda}$ is a JMT subring of $T$.
\par Let
\[
A = \bigcup_{\alpha \in \Omega} R_\alpha.
\]
Then $(A, A \cap M)$ is a Min-T subring of $T$.
\par Note that $M^2 \not \subseteq P$ for every $P \in \C$, so the map $A \to T / M^2$ is onto. Next, let $I = (a_1, \ldots, a_n)A$ be a finitely generated ideal of $A$ and $c \in IT \cap A$. Then, for some $\delta \in \Omega$ such that $\delta$ has a predecessor, we have $\{ c, a_1, \ldots, a_n \} \subseteq R_{\delta}$. In particular, this yields $c \in (a_1,a_2, \ldots ,a_n)R_{\delta} \subseteq I$. Hence, $IT \cap A = I$ for all finitely generated ideals $I$ of $A$. Thus $A$ is Noetherian and $\widehat{A} = T$.
\par Since $T$ is faithfully flat over $A$, the ideals $\C_i \cap A$ are the minimal prime ideals of $A$. By construction, for every ideal $J$ not contained in any $P \in \C$, $A$ contains an element of $J$ not in any minimal prime ideal of $T$.  Suppose for such an ideal $J$ that $J \cap A$ is a minimal prime ideal of $A$.  Then $J \cap A = Q \cap A$ for some minimal prime ideal $Q$ of $T$.  This contradicts that $A$ contains an element of $J$ not in any minimal prime ideal of $T$.  It follows that the formal fiber of $\C_i \cap A$ has countably many maximal elements, which are precisely the elements of $\C_i$. Furthermore, by construction, the map $A \to T / J$ is onto for any ideal $J$ such that $J \not \subseteq P$ for all $P \in \C$.
\par Now suppose $p \in \Spec(A)$ with $\text{ht}(p) > 0$. Then $pT \not \subseteq P$ for every $P \in \C$, and so the map $A \to T / pT$ is onto. Since $pT \cap A = p$, $A/p \cong T / pT$. It follows that the only element in the formal fiber of $p$ is $pT$. Thus we have a bijection between $\{J \in \Spec(T) \mid J \not \subseteq P \text{ for every } P \in \C \}$ and $\{j \in \Spec(A) \mid j \not \subseteq q \text{ for every } q \in \Min(A) \}$. Furthermore, by Theorem \ref{success} this bijection implies, for all $X \subseteq \{ 1, \ldots, m \}$, a bijection between $\mathcal{S}_T(X)$ and $\mathcal{S}_A(X)$.
\end{proof}
\hfill \newline

\begin{rmk}
Although the conditions of the previous theorem are complicated, they imply the following:
Let $(T, M)$ be a complete local ring of dimension at least one and let $\P = (\C, \{ \C_{i = 1}^m \})$ be a minfeasible partition and assume that it is not the case that one of the following conditions hold:
\begin{enumerate}[label={(\roman*)}]
\item $\textup{char}(T) = 0$, $M \cap \mathbb{Z} = (p)$, and one of the following is true:
\begin{enumerate}
        \item There exist prime ideals $P,P'$ contained under the same $\C_i$ such that $p \in P$ but $p \notin P'$.
        \item For some prime ideal $Q$ under some $\C_i$ (where $\C_i \neq \{Q\}$), $p \in Q$ and $\textup{Ann}_T(p) \subseteq Q$.
\end{enumerate}
\item $\textup{char}(T) = p^k$, $ k \geq 2$ for some prime $p$ and $\textup{Ann}_T(p) \subseteq Q$ for some $Q$ under $\C_i$ (where $\C_i \neq \{Q\}$).
\end{enumerate}
Then $T$ is the completion of a local ring $A$ such that \begin{enumerate}[label={(\roman*)}]
    \item $\Min(A) = \{ \C_1 \cap A, \ldots, \C_m \cap A \}$, and the formal fiber of $\C_i \cap A$ has countably many maximal elements, which are precisely the elements of $\C_i$.
    \item If $J$ is an ideal of $T$ such that $J \subsetneq P$ for every $P \in \C$, then the natural map $A \rightarrow T/J$ is onto so that there is a bijection between $\{ J \in \Spec(T) \mid J \not \subseteq P \text{ for every } P \in \C \}$ and $\{ j \in \Spec(A) \mid j \not \subseteq q \text{ for every } q \in \Min(A) \}$.
    \item Using the notation of Theorem \ref{success}, for all $X \subseteq \{ 1, \ldots, m \}$ we have a bijection between $\mathcal{S}_T(X)$ and $\mathcal{S}_A(X)$.
\end{enumerate}
\end{rmk}

Note that the bijection between $\S_T(X)$ and $\S_A(X)$ implies a similar bijection between $\S_T^1(X)$ and $\S_A^1(X)$. Under certain conditions, we can construct a countable precompletion (with countable spectrum) with some of the same control of formal fibers of the minimal prime ideals. We can not preserve the bijection between $\mathcal{S}_T(X)$ and $\mathcal{S}_A(X)$, but we can preserve the bijection between $\S_T^1(X)$ and $\S_A^1(X)$. To do so, we need to do some extra work in constructing the base ring. 
\par First, we have a preliminary lemma about residue field powers. Though the result is well-known, we include a proof for completeness. \\

\begin{lem} \label{residue field powers}
    Let $(A,M)$ be a local ring. If $A/M$ is finite, then $A/M^n$ is finite for all $n$. If $A/M$ is infinite, then $|A/M^n| = |A/M|$ for all $n$.
\end{lem}

\begin{proof}
    We have that
    \[
    |A/M^n| = |A/M|\cdot|M/M^2| \cdots |M^{n-1}/M^n| = |A/M|\cdot |A/M|^{a_2} \cdots |A/M|^{a_n}
    \]
    Where $a_i$ is the dimension of $M^{i-1}/M^i$ as a vector space over $A/M$. Since $A$ is Noetherian, each $a_i$ is finite and so the result follows.
\end{proof}

The following lemma ensures that the intersection of distinct prime ideals in $\S_T^1(X)$ with our countable JMT subring remain distinct.

\begin{lem}\label{countable set-up}
Let $(T, M)$ and $\P = (\C, \{ \C_i \}_{i = 1}^m)$ be as in Remark \ref{conditions}. Then there exists a countable JMT subring $\overline{R}$ of $T$ such that for all $X \subseteq \{ 1, \ldots, m \}$ with $|X| \ge 2$ 
and for any $P, P' \in \mathcal{S}_T^1(X)$, $P \cap \overline{R}$ and $P' \cap \overline{R}$ are incomparable.
\end{lem}

\begin{proof}
\par Let $X \subseteq \{ 1,\ldots,m \}$ with $|X| \ge 2$. Every  prime ideal $P \in \mathcal{S}_T^1(X)$ is minimal over a sum $\sum_{i \in X} Q_i$ for $Q_i \in \Min(T)$ where $|X| \ge 2$, so there can only be finitely many such ideals $P$. Since $Q_i$ is in $\Min(T)$ for every $i$, there are only finitely many possible sets, and thus $\mathcal{S}_T^1(X)$ is finite. 
\par Let $R$ be the JMT subring constructed in Lemma \ref{Byron 2.13}, and note that $R$ is countable because $R$ is obtained by adjoining finitely many elements to the prime subring, which is countable. Fix $X$ and let $\mathcal{S}_T^1(X)$ be $\{ P_1, \ldots, P_n \}$. Then fix $k \in \{ 1, \ldots, n \}$. The elements of $\S_T^1(X)$ are incomparable by definition, so we can pick an element $p_k \in P_k$ such that $p_k \notin P_j$ for any $j \neq k$ and $p_k \notin P$ for any $P \in \C$. We will show that there exists a unit $u$ such that $R'= R[up_k]_{R[up_k] \cap M}$ is a JMT subring of $T$ with $P_k \cap R' \neq P_j \cap R'$ for any $j \neq k$.
\par This part of the proof is similar to the proof of Lemma 3.10 in \cite{small09}. Define $S = R[p_k]$, and note that $|S| = |R| < |T|$. Apply Lemma \ref{Byron 2.2} with $t = 0,q = 1$, and $J = M$ to find an element $t' \in M$ such that, for each $P \in \C$, $t' + P \in T/P$ is transcendental over $S/(P \cap S)$. Let $u = t' + 1$. Note that $u$ is a unit, and that $u + P \in T/P$ is transcendental over $S/(P \cap S)$, because $1 \in S$. Define $p = up_k$. Then $p \in P_k$, $p \notin P_j$ for any $j \neq k$, and $p_k \notin P$ for any $P \in \C$. Let $R' = R[p]_{R[p] \cap M}$ and note that $P_k \cap R' \neq P_j \cap R'$ for any $j \neq k$ and $P_k \cap R' \neq P \cap R'$ for any $P \in \C$.
\par Now we show that $p + P \in T/P$ is transcendental over $R/(P \cap R)$ for every $P \in \C$. By Lemma \ref{old 3.3}, this implies conditions (ii), (iii), and (iv) of Definition \ref{Min-T subring} hold for $R[p]$. Let $f \in R[p] \cap P$ for some $P \in \C$. Then
\begin{align}
    \notag f &= r_n p^n + \cdots + r_1 p + r_0 \\
    \notag &= r_n(p_ku)^n + \cdots + r_1(p_k u) + r_0 \\
    \notag &= (r_n p_k^n)u^n + \cdots + (r_1 p_k)u + r_0.
\end{align}
Hence we can express $f$ as an element of $S[u] \cap P$. Since $u + P \in T/P$ is transcendental over $S/(P \cap S)$, this implies that $r_i p_k^i \in P$, and thus $r_i \in P$ for all $i = 1, \ldots, n$. Hence, $p + P \in T/P$ is transcendental over $R/(P \cap R)$. It follows that $R' = R[p]_{R[p] \cap M}$ is a JMT subring of $T$.
\par Now repeat the process for each $k \in \{ 1, \ldots ,n \}$.
In this way, we obtain a JMT subring $R''$ and elements $p_1, p_2, \ldots ,p_n$ such that, for every $i = 1,2, \ldots ,n$, $p_i \in P_i \cap R''$ and $p_i \not\in P_j \cap R''$ when $j \neq i$.  Starting with $R''$, repeat this process for another choice of $X$, and continue until all possible sets $X$ are exhausted to obtain a JMT subring $\overline{R}$ of $T$.  It follows that, for all $X \subseteq \{ 1, \ldots, m \}$ with $|X| \ge 2$ and for any $P, P' \in \mathcal{S}_T^1(X)$, $P \cap \overline{R}$ and $P' \cap \overline{R}$ are incomparable. Then, observe that $\overline{R}$ was obtained by adjoining finitely many elements to a countable subring, and hence $\overline{R}$ is countable.
\end{proof}

We first state a general version of a theorem, and then the countable case as a corollary.

\begin{thm} \label{countable case}
Let the setup be as in Theorem \ref{main theorem}. Then $T$ is the completion of a subring $A$ such that
\begin{enumerate}[label={(\roman*)}]
    \item If $T/M$ is infinite, then $|A| = |T/M|$. If $T/M$ is finite, then $A$ is countable.
    \item $\Min(A) = \{ \C_1 \cap A, \ldots, \C_m \cap A \}$
    \item For all $X \subseteq \{ 1, \ldots, m \}$ with $|X| \ge 2$, there is a bijection between $\S_T^1(X)$ and $\S_A^1(X)$.
\end{enumerate}
\end{thm}

\begin{proof}
The proof will be similar to the proof of Theorem \ref{main theorem}. By Lemma \ref{residue field powers}, if $T/M$ is infinite, then $|T/M| = |T/M^2|$ and otherwise $T/M^2$ is finite. Let $\Omega = T / M^2$. Well-order $\Omega$ so that every element has strictly fewer than $|\Omega|$ predecessors.
\par For each $\alpha \in \Omega$, we let $|\alpha| = |\{ \beta \in \Omega \mid \beta \leq \alpha \}|$. Let 0 denote the first element of $\Omega$, and let $R_0$ be the JMT subring constructed in Lemma \ref{countable set-up}. For each $\lambda \in \Omega$ after the first, we define $R_{\lambda}$ recursively as follows: assume $R_{\beta}$ is defined for all $\beta < \lambda$ such that $R_{\beta}$ is a JMT subring, and $|R_{\beta}| \leq |\beta||R_0|$ for all $\beta < \alpha$. As before, let $\gamma(\lambda) = u + J$ denote the least upper bound of the set of predecessors of $\lambda$. If $\gamma(\lambda) < \lambda$, we use Lemma \ref{old 3.8} with $R = R_{\gamma(\lambda)}$ to find a JMT subring $R_{\lambda}$ such that
\begin{enumerate}[label={(\roman*)}]
    \item $R_{\gamma(\lambda)} \subseteq R_{\lambda} \subseteq T$;
    \item if $u \in J$, then $R_{\lambda} \cap J$ contains an element $t \notin \bigcup_{Q \in \Min(T)} Q$;
    \item the coset $\gamma(\lambda) = u+M^2$ is in the image of the map $R_{\lambda} \rightarrow T/M^2$;
    \item for every finitely generated ideal $I$ of $R_{\lambda}$, we have $IT \cap R_{\lambda} = I$;
\end{enumerate}
In this case,
\[
|R_{\lambda}| = |R_{\gamma(\lambda)}| \leq |\gamma(\lambda)||R_0| \leq |\lambda||R_0|
\]
On the other hand, if $\gamma(\lambda) = \lambda$, we let
\[
R_{\lambda} = \bigcup_{\beta < \lambda} R_\beta.
\]
Then $|\lambda| < |\Omega||R_0| = |T/M^2||R_0| \leq |T|$, and $|R_{\lambda}| \leq |\lambda||R_0|$. By Lemma \ref{old 3.1}, $R_{\lambda}$ is a JMT subring of $T$.
\par In either case, if $T/M$ is infinite then $|R_{\lambda}| \leq |T/M|$ and otherwise $R_{\lambda}$ is countable.
\par Let
\[
A = \bigcup_{\alpha \in \Omega} R_\alpha.
\]
Then $(A, A \cap M)$ is a Min-T subring of $T$.
\par Note that the map $A \to T / M^2$ is onto by construction. Next, let $I = (a_1, \ldots, a_n)A$ be a finitely generated ideal of $A$ and $c \in IT \cap A$. Then, for some $\beta \in \Omega$, where $\beta$ has a predecessor, we have $\{ c, a_1, \ldots, a_n \} \subseteq R_{\beta}$. In particular, this yields $c \in (a_1, \dots ,a_n)R_{\beta} \subseteq I$. Hence, $IT \cap A = I$ for all finitely generated ideals $I$ of $A$. By Proposition \ref{cpm}, $A$ is Noetherian and $\widehat{A} = T$. Moreover, since $T$ is faithfully flat over $A$, the ideals $\C_i \cap A$ are the minimal prime ideals of $A$.
\par If $T/M$ is finite, then $A$ is a finite union of countable subrings, so $A$ is countable. If $T/M$ is infinite, then $A$ is a union of cardinality $|T/M|$ of subrings with cardinality $\leq |T/M|$, so $|A| \leq |T/M|$, and thus the isomorphism $A/(M \cap A) \cong T/M$ implies that $|T/M| = |A/(M \cap A)| \leq |A| \leq |T/M|$, so $|A|=|T/M|$.
\par Now, note that for all $X \subseteq \{ 1, \ldots, m \}$ with $|X| \ge 2$, $P \cap R_0$ and $P' \cap R_0$ are incomparable for any distinct $P, P' \in \S_T^1(X)$, so $P \cap A$ and $P' \cap A$ are incomparable for any distinct $P, P' \in \mathcal{S}_T^1(X)$. 
Let $P \in \S_T^1(X)$. We first show that $P \cap A \in \mathcal{S}_A^1(X)$.  Note that $P \in \S_T(X)$, so $P \cap A \in \S_A(X)$, and in particular $P \cap A \supseteq \sum_{i \in X} \C_i \cap A$ so that $P \cap A$ contains a minimal prime ideal $P' \cap A$ over $\sum_{i \in X} \C_i \cap A$. 
Then $P' \cap A \in \S_A^1(X)$, so that we can assume $P' \in \S_T^1(X)$ by Corollary \ref{height-one map}. 
But then $P,P'$ are in $\S_T^1(X)$ so $P \cap A \supseteq P' \cap A$ implies $P = P'$ and it follows that $P \cap A \in \mathcal{S}_A^1(X)$.
Note that the map $P \mapsto P \cap A$ is a surjection from $\mathcal{S}_T^1(X)$ to $\mathcal{S}_A^1(X)$ and that, because $P \cap A \neq P' \cap A$ for any distinct $P, P' \in \mathcal{S}_T^1(X)$, this map is a bijection.
\end{proof}

\begin{cor} \label{countable case cor}
Let the setup be as in Theorem \ref{main theorem}. Suppose that $T/M$ is countable. Then $T$ is the completion of a countable subring $A$ such that
\begin{enumerate}[label={(\roman*)}]
    \item $\Min(A) = \{ \C_1 \cap A, \ldots, \C_m \cap A \}$
    \item For all $X \subseteq \{ 1, \ldots, m \}$ with $|X| \ge 2$, there is a bijection between the sets $\S_T^1(X)$ and $\S_A^1(X)$
\end{enumerate}
\end{cor}

\begin{rmk}
Note that given any complete local ring $(T,M)$ with minimal prime ideals $Q_1,...,Q_n$, we can construct a minfeasible partition on $T$ given by $\C_i = \{ Q_i \}$ for all $i$. Then, by Theorem \ref{countable case}, we see that, if $T/M$ is infinite, $(T,M)$ has a precompletion $A$ with $n$ minimal prime ideals where $|A|=|T/M|$, and if $T/M$ is finite, then $(T,M)$ has a precompletion $A$ with $n$ minimal prime ideals where $A$ is countable.  
\end{rmk}

\begin{ex}\label{standardexample}
Consider the complete local ring $\frac{\mathbb{Q}[[x,y,z]]}{( xyz )}$ from Example \ref{spec diagram ex}, and the minfeasible partition $\C =  \{ ( x ), ( y ), ( z ) \}$, $\C_1 =  \{ ( x ), ( y ) \}$, and $\C_2 =  \{( z ) \}$ from Example \ref{minfeasible ex}. Using Theorem \ref{main theorem} and Corollary \ref{countable case cor}, we can construct a countable precompletion $A$ and uncountable precompletion $A'$ of $\frac{\mathbb{Q}[[x,y,z]]}{( xyz )}$ such that $A$ and $A'$ have two minimal prime ideals and where $|\S_A^1(1,2)|=2= |\S_{A'}^1(1,2)|$. The Spec diagrams of $A$ and $A'$ are below:

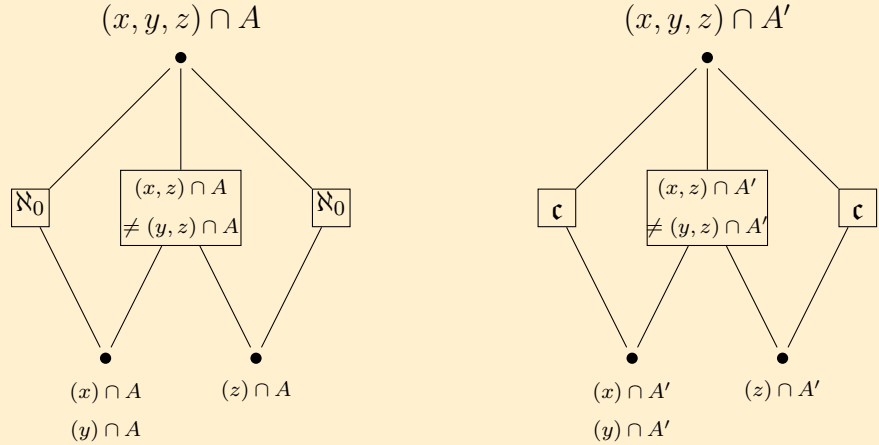
\begin{center}

  \begin{tikzpicture}
        \node[label=above:{$( x,y,z ) \cap A$}] (A) at (3, 3) {};
        \node[shape=rectangle,draw=black,minimum size=0.5cm,label={[label distance=-0.5cm,align=center]\small$\aleph_0$}] (B) at (1, 1) {};
        \node[shape=rectangle,draw=black,minimum height=1.0cm,minimum width =1.6cm,label={[label distance=-1.05cm,align=center]\scriptsize$( x,z ) \cap A$  \\ \scriptsize$\neq ( y,z ) \cap A$}] (C) at (3, 1) {};
        \node[shape=rectangle,draw=black,minimum size=0.5cm,label={[label distance=-0.5cm,align=center]\small$\aleph_0$}] (D) at (5, 1) {};
        \node[label={[label distance=-1.4cm,align=center]\scriptsize $( x ) \cap A$ \\ \scriptsize $( y ) \cap A$}] (E) at (2, -1) {};
        \node[label=below:\scriptsize $( z ) \cap A$] (F) at (4, -1) {};
            
        \foreach \n in {A,E,F}
            \node at (\n)[circle,fill,inner sep=1.5pt]{};
        
        \path [-] (A) edge node[left] {} (B);
        \path [-] (A) edge node[left] {} (C);
        \path [-] (A) edge node[left] {} (D);
        \path [-] (B) edge node[left] {} (E);
        \path [-] (C) edge node[left] {} (E);
        \path [-] (C) edge node[left] {} (F);
        \path [-] (D) edge node[left] {} (F);

        \node[label=above:{$( x,y,z ) \cap A'$}] (A) at (10, 3) {};
        \node[shape=rectangle,draw=black,minimum size=0.5cm,label={[label distance=-0.55cm,align=center]$\mathfrak{c}$}] (B) at (8, 1) {};
        \node[shape=rectangle,draw=black,minimum height=1.0cm,minimum width =1.6cm,label={[label distance=-1.05cm,align=center]\scriptsize$( x,z ) \cap A'$  \\ \scriptsize$\neq ( y,z ) \cap A'$}] (C) at (10, 1) {};
        \node[shape=rectangle,draw=black,minimum size=0.5cm,label={[label distance=-0.55cm,align=center]$\mathfrak{c}$}] (D) at (12, 1) {};
        \node[label={[label distance=-1.4cm,align=center]\scriptsize $( x ) \cap A'$ \\ \scriptsize $( y ) \cap A'$}] (E) at (9, -1) {};
        \node[label=below:\scriptsize $( z ) \cap A'$] (F) at (11, -1) {};
            
        \foreach \n in {A,E,F}
            \node at (\n)[circle,fill,inner sep=1.5pt]{};
        
        \path [-] (A) edge node[left] {} (B);
        \path [-] (A) edge node[left] {} (C);
        \path [-] (A) edge node[left] {} (D);
        \path [-] (B) edge node[left] {} (E);
        \path [-] (C) edge node[left] {} (E);
        \path [-] (C) edge node[left] {} (F);
        \path [-] (D) edge node[left] {} (F);

    \end{tikzpicture}
\end{center}
\end{ex}

\section{Maximal Chains}
A ``maximal chain" in a local ring $(T,M)$ is a saturated chain of prime ideals whose minimal element is a minimal prime ideal of $T$ and whose maximal element is the maximal ideal $M$ of $T$. In this section, we examine the relationship between maximal chains in a complete local ring $T$ and maximal chains in precompletions of $T$. We also examine the behavior of maximal chains under the construction of the previous section.
\par We start by building on the following result from \cite{small17}.

\begin{lem}[\cite{small17} Lemma 2.7] \label{2017 2.7}
Let $(T,M)$ be a complete local ring and let $A$ be a local domain such that $\widehat{A} = T$. Suppose $\mathcal{C}_T$ is a chain of prime ideals of $T$ of the form $P_0 \subsetneq \ldots \subsetneq P_{n - 1} \subsetneq M$, where $n \ge 2$ and $P_0 \cap A = (0)$. If the chain $\mathcal{C}_A$ given by $(0) = P_0 \cap A \subseteq \ldots \subseteq P_{n - 1} \cap A \subseteq M \cap A$ has length less than $n$, then $\mathcal{C}_A$ is not saturated.
\end{lem}

The next lemma generalizes the previous lemma for non-domain precompletions.

\begin{lem} \label{chain length}
Let $(T,M)$ be a complete local ring and let $A$ be a local ring such that $\widehat{A} = T$. Suppose $\mathcal{C}_T$ is a chain of prime ideals of $T$ of the form $P_0 \subsetneq \ldots \subsetneq P_{n - 1} \subsetneq M$, where $n \ge 2$ and $P_0 \cap A \in \Min(A)$. If the chain $\mathcal{C}_A$ given by $P_0 \cap A \subseteq \ldots \subseteq P_{n - 1} \cap A \subseteq M \cap A$ has length less than $n$, then $\mathcal{C}_A$ is not saturated.
\end{lem}
\begin{proof}
Examine the chains
\[
\frac{P_0}{(P_0 \cap A)T} \subsetneq \ldots \subsetneq \frac{P_{n-1}}{(P_0 \cap A)T} \subsetneq \frac{M}{(P_0 \cap A)T}
\]
and
\[
(0) = \frac{P_0 \cap A}{P_0 \cap A} \subseteq \ldots \subseteq \frac{P_{n-1} \cap A}{P_0 \cap A} \subseteq \frac{M \cap A}{P_0 \cap A}
\]
in the rings
\[
\frac{A}{P_0 \cap A} \text{ and } \widehat{\frac{A}{P_0 \cap A}} = \frac{T}{(P_0 \cap A)T}
\]
The first chain still has length $n$, and the second chain has length less than $n$, so by Lemma \ref{2017 2.7} the second chain is not saturated. Therefore, the chain $\mathcal{C}_A$ given by $P_0 \cap A \subseteq \ldots \subseteq P_{n - 1} \cap A \subseteq M \cap A$ is not saturated.
\end{proof}

The following lemma essentially states that chain length cannot decrease under the natural map from $\Spec(T)$ to $\Spec(A)$.

\begin{lem} \label{chain lower bound} 
Let $(T,M)$ be a complete local ring and $A$ a local ring with $\widehat{A} = T$. Let $\mathcal{C}_T$ be a chain of prime ideals of $T$ of the form $P_0 \subsetneq \ldots \subsetneq P_{n - 1} \subsetneq M$ with $P_0 \cap A \in \Min(A)$. Then there exists a chain $\mathcal{C}_T'$ in $T$ of the form $P_0' \subsetneq \ldots \subsetneq P_{n - 1}' \subsetneq M$ with length $n$ and $P_0' \cap A \in \Min(A)$ such that the chain $\mathcal{C}_A'$ given by $P_0' \cap A \subsetneq \ldots \subsetneq P_{n - 1}' \cap A \subsetneq M \cap A$ is a chain of length $n$. Moreover, the chain $\mathcal{C}_A$ given by $P_0 \cap A \subseteq \ldots \subseteq P_{n - 1} \cap A \subseteq M \cap A$ is contained in $\mathcal{C}_A'$, so that $\mathcal{C}_A$ is contained in a saturated chain of length at least $n$.
\end{lem}

\begin{proof}
Examine the chain $\mathcal{C}_A$ given by $P_0 \cap A \subseteq \ldots \subseteq P_{n - 1} \cap A \subseteq M \cap A$. If $\mathcal{C}_A$ has length $n$, we are done. If $\mathcal{C}_A$ has length $m<n$, then by Lemma \ref{chain length}, $\mathcal{C}_A$ is not saturated, so we can insert some prime ideal $P^* \cap A$ into the chain and increase its length. Simultaneously we can remove a prime ideal $P_j$ such that $P_j \cap A = P_{j+1} \cap A$. Then we have the chain (note the order of $j,k$ is ambiguous):
\[
P_0 \cap A \subseteq \ldots\subseteq P_{k-1} \cap A \subsetneq P^* \cap A \subsetneq P_k \cap A \subseteq \ldots \subseteq P_{j-1} \cap A \subseteq P_{j+1} \cap A \subseteq \ldots \subseteq P_{n - 1} \cap A \subseteq M \cap A
\]
of length $m + 1$. If $P^* \nsubseteq P_k$, then by Lemma \ref{General Poset Property} there exists a prime ideal $P^{**}$ in T such that $P^{**} \subseteq P_k$ and $P^{**} \cap A = P^* \cap A$. 
Then, again by Lemma \ref{General Poset Property}, if $P^{**} \nsupseteq P_{k - 1}$, there exists a prime ideal $P_{k - 1}'$ in T such that $P_{k - 1}' \subseteq P^{**}$ and $P_{k - 1}' \cap A = P_{k - 1} \cap A$.
Continuing in this manner, we construct a chain $\mathcal{C}_T'$ in $T$ of the form $P_0' \subsetneq \ldots \subsetneq P_{n - 1}' \subsetneq M$ of length $n$ such that the chain $\mathcal{C}_A'$ given by $P_0' \cap A \subseteq \ldots \subseteq P_{n - 1}' \cap A \subseteq M \cap A$ has length $m + 1$. 
If $m + 1 < n$, repeat the above process to eventually obtain a chain $\mathcal{C}_T'$ in $T$ of the form $P_0' \subsetneq \ldots \subsetneq P_{n - 1}' \subsetneq M$ of length $n$ such that the chain $\mathcal{C}_A'$ given by $P_0' \cap A \subsetneq \ldots \subsetneq P_{n - 1}' \cap A \subsetneq M \cap A$ has length $n$. 
Note that $\mathcal{C}_A'$ is constructed so that $\mathcal{C}_A$ is contained in $\mathcal{C}_A'$, and hence $\mathcal{C}_A$ is contained in a saturated chain of length at least $n$.
\end{proof}

The next lemma states that given a chain of prime ideals of length $n$ in $A$, it must have come from a chain of prime ideals of length at least $n$ in $T$.
\begin{lem} \label{chain upper bound}
Let $(T, M)$ be a complete local ring and $A$ a local ring such that $\widehat{A} = T$. Let $P_0 \cap A \subsetneq \ldots \subsetneq P_{n - 1} \cap A \subsetneq M \cap A$ where $P_i \in \Spec (T)$ be a chain of prime ideals in $A$ of length $n$ where $P_0 \cap A \in \Min(A)$. Then there exists a chain of length $n$ in $T$ of the form $P_0' \subsetneq \ldots \subsetneq P_{n - 1}' \subsetneq M$, such that $P_0' \in \Min(T)$ and $P_i' \cap A = P_i \cap A$ for all $i$.
\end{lem}

\begin{proof}
Start with the set $\{ P_0, \ldots, P_{n - 1}, M \}$. It is clear that $P_{n - 1} \subsetneq M$. Let $P_{n - 1}' = P_{n - 1}$. If $P_{n - 2} \nsubseteq P_{n - 1}'$, then by Lemma \ref{General Poset Property}, we can find $P_{n - 2}'$ such that $P_{n - 2}' \cap A = P_{n - 2} \cap A$ and $P_{n - 2}' \subseteq P_{n - 1}'$. Because $P_{n - 2}' \cap A = P_{n - 2} \cap A \neq P_{n-1} \cap A = P_{n-1}' \cap A$, we have that $P_{n - 2}' \subsetneq P_{n - 1}'$. Continuing in this manner, we can obtain the desired chain except that $P_0'$ is not necessarily in $\Min (T)$. Then note that $P_0'$ must contain some $Q \in \Min (T)$, and that $P_0' \cap A = Q \cap A$, so we can replace $P_0$ with $Q$ to obtain the desired chain.
\end{proof}

The following theorem states that a maximal chain of length $n$ can only exist in the precompletion of $T$ if there exists a chain of length $n$ in $T$. In addition, using the results of the previous section, we can construct precompletions that control the minimal prime ideals in the manner of Theorem \ref{main theorem} and preserve all of the chain lengths of $T$.

\begin{thm} \label{chain preservation thm}
Let $(T,M)$ be a complete local ring of dimension at least two with $n$ minimal prime ideals $\{ Q_1, \ldots ,Q_n \}$ such that the coheight of $Q_i$ is $c_i$. Then, for every subring $A$ of $T$ such that $\widehat{A} = T$
\begin{enumerate}[label={(\roman*)}]
    \item every saturated chain $P_0 \cap A \subsetneq \ldots \subsetneq P_{n - 1} \cap A \subsetneq M \cap A$ with $P_i \in \Spec (T)$ in $A$ such that $P_0 \cap A \in \Min(A)$ has length $c_i$ for some $i$. 
    \item Every minimal prime ideal $q$ of $A$ has coheight equal to $\textup{max}\{\textup{coht} (Q) \mid Q \in \Min(T), Q \cap A = q \}$.
    \end{enumerate}
Furthermore, given any minfeasible partition $\mathcal{P} = (\C, \{ \C_i \}_{i = 1}^{m})$ on $T$, if $T$ satisfies the conditions of Theorem \ref{main theorem}, there exists a local subring $A$ of $T$ with $\widehat{A} = T$ such that:
\begin{enumerate}[label={(\roman*)}]
    \item $\Min(A) = \{ \C_1 \cap A, \ldots, \C_m \cap A \}$, and the formal fiber of $\C_i \cap A$ has countably many maximal elements, which are precisely the elements of $\C_i$
    \item If $J$ is an ideal of $T$ such that $J \subsetneq P$ for every $P \in \C$, then the natural map $A \rightarrow T/J$ is onto so that there is a bijection between $\{ J \in \Spec(T) \mid J \not \subseteq P \text{ for every } P \in \C \}$ and $\{ j \in \Spec(A) \mid j \not \subseteq q \text{ for every } q \in \Min(A) \}$.
    \item Using the notation of Theorem \ref{success}, for all $X \subseteq \{ 1, \ldots, m \}$ we have a bijection between $\mathcal{S}_T(X)$ and $\mathcal{S}_A(X)$.
    \item For every $Q_i$, there is a chain $Q_i \subsetneq P_1  \subsetneq \ldots \subsetneq P_{c_i-1} \subsetneq M$ of prime ideals of $T$ such that the chain $Q_i \cap A \subsetneq P_1 \cap A \subsetneq \ldots \subsetneq P_{c_i-1} \cap A \subsetneq M \cap A$ has length $c_i$.
\end{enumerate}
\end{thm}
\begin{proof}
Let the chain $P_0 \cap A \subsetneq \ldots \subsetneq P_{m - 1} \cap A \subsetneq M \cap A$ where $P_i \in \Spec(T)$ and $P_0 \cap A \in \Min(A)$ be a saturated chain in $A$ of length $m$. 
We will show that $m = c_i$ for some $i$. By Lemma \ref{chain upper bound}, there is a chain $P_0' \subsetneq \ldots \subsetneq P_{m - 1}' \subsetneq M$ of prime ideals of $T$ of length $m$ such that $P_0' \in \Min(T)$ and $P_j' \cap A = P_j \cap A$ for all $j$. 
This chain is contained in a saturated chain $P_0' \subsetneq \ldots \subsetneq P_{c_i - 1}' \subsetneq M$, which must have length $c_i \geq m$ for some $i$ because $P_0' \in \Min(T)$ and $T$ is catenary. 
Then the chain $P_0' \cap A \subseteq \ldots \subseteq P_{c_i - 1}' \cap A \subsetneq M \cap A$ contains the chain $P_0 \cap A \subsetneq \ldots \subsetneq P_{m - 1} \cap A \subsetneq M \cap A$ and is contained in a chain of length $c_i$ by Lemma \ref{chain lower bound}. 
But $P_0 \cap A \subsetneq \ldots \subsetneq P_{m - 1} \cap A \subsetneq M \cap A$ is saturated, so $m = c_i$.

Now let $q \in \Min(A)$ and let $Q \in \Min(T)$ be such that $Q \cap A = q$ and $Q$ has maximal coheight among the elements of the formal fiber of $q$. 
Let $\cht(Q) = k$. Then there exists a chain $Q = P_0 \subsetneq \ldots \subsetneq P_{k - 1} \subsetneq M$ in $T$ of length $k$. 
Then by Lemma \ref{chain lower bound}, there exists a chain $P_0' \subsetneq \cdots \subsetneq P_{k-1}' \subsetneq M$ in $T$ such that the chain $q = P_0' \cap A \subsetneq \ldots \subsetneq P_{k - 1}' \cap A \subsetneq M \cap A$ has length $k$, so that $\cht(q) \geq k$. 
Now assume that $\cht(q) > k$. Then by Lemma \ref{chain upper bound}, there exists a chain $P_0' \subsetneq \ldots \subsetneq P_{l - 1}' \subsetneq M$ of length $\cht(q)$, such that $P_0' \in \Min (T)$ and $P_0' \cap A = q$ so that $\cht(P_0') > k$, a contradiction. Thus $\cht(q) = k = \cht(Q)$.

For the second claim, note that the ring constructed in Theorem \ref{main theorem} already satisfies properties (i)-(iii), and the bijection in (ii) implies (iv) as follows: Let $Q_i \in \Min(T)$. $Q_i$ has coheight $c_i$ so there exists a chain $Q_i \subsetneq P_1  \subsetneq \ldots \subsetneq P_{c_i-1} \subsetneq M$ of prime ideals of $T$ of length $c_i$. Then by Lemma \ref{chain lower bound} there exists a chain $P_0' \subsetneq P_1'  \subsetneq \ldots \subsetneq P_{c_i-1}' \subsetneq M$ of prime ideals of $T$ such that $P_0' \cap A = Q_i \cap A$ and $P_i' \cap A = P_i \cap A$ for all $i$ and such that $P_0' \cap A \subsetneq P_1' \cap A \subsetneq \ldots \subsetneq P_{c_i-1}' \cap A \subsetneq M \cap A$ is a chain of length $c_i$. Then the bijection in (ii) along with Lemma \ref{General Poset Property} implies that the chain is saturated.
\end{proof}

As in the last section, we can preserve chains and construct a countable precompletion of $T$. We will state this corollary generally for all cardinalities, but it can be simplified to the case where $A$ is countable.

\begin{cor} \label{countable chain thm}
Let $(T, M)$ be a complete local ring.
Let $\P = (\C, \{ \C_i \}_{i = 1}^{m})$ be a minfeasible partition and define 
\[
c_i = \textup{max}(\{\textup{\cht}(Q) \, | \, Q \textup{ is under }\C_i\})
\]
If $T$ satisfies the conditions of Remark \ref{conditions}, then there exists a local subring $A$ of $T$ such that $\widehat{A} = T$ and:
\begin{enumerate}[label={(\roman*)}]
    \item If $|T/M|$ is infinite, then $|A| = |T/M|$. If $|T/M|$ is finite, then $A$ is countable.
    \item $\Min(A) = \{ \C_1 \cap A, \ldots, \C_m \cap A \}$
    \item For all $X \subseteq \{ 1, \ldots, m \}$ with $|X| \ge 2$, there is a bijection between $\S_T^1(X)$ and $\S_A^1(X)$.
    \item $\C_i \cap A$ has coheight $c_i$ for all $i$, so that, in particular, $A$ has maximal chains of prime ideals of length $c_i$ for all $i$.
    \item If $\C_i \cap A \subsetneq ... \subsetneq M \cap A$ is a maximal chain of prime ideals of $A$, then its length is equal to $\cht(Q)$ for some minimal prime ideal $Q$ under $\C_i$. 
\end{enumerate}
\end{cor}
\begin{proof}
Use the minfeasible partition above and Theorem \ref{countable case} to construct a precompletion $A$ of $T$ that satisfies conditions (i)-(iii).
Then by
Theorem \ref{chain preservation thm}, every minimal prime ideal $\C_i \cap A$ of $A$ has coheight $c_i$, proving (iv).
Condition (v) follows from Theorem \ref{chain preservation thm}. 
\end{proof}

\section*{Acknowledgements}
We would like to thank the National Science Foundation (grant DMS-1659037); Williams College; and the Clare Boothe Luce Program for providing funding for this research.

\bibliographystyle{plain}
\bibliography{references}

\end{document}